\documentclass[a4paper,12pt]{amsart}

\usepackage[utf8]{inputenc}
\usepackage[T1]{fontenc}
\usepackage{lmodern}
\usepackage[english]{babel}
\usepackage{microtype}
\usepackage{amsmath}
\usepackage{amssymb}
\usepackage{mathrsfs}
\usepackage{amsthm}
\usepackage{esint}
\usepackage[abbrev]{amsrefs}
\usepackage{cleveref}
\usepackage{tikz} 
\usepackage{pgfplots}

\makeatletter
\def\namedlabel#1#2{\begingroup
   \def\@currentlabel{#2}%
   \label{#1}\endgroup
}
\makeatother

\usepackage{pgf,tikz}
\usetikzlibrary{arrows}
\usetikzlibrary[patterns]

\usepackage[left=2.5cm,right=2.5cm,top=4.5cm,bottom=3.5cm]{geometry}
\theoremstyle{plain}
\newtheorem{proposition}{Proposition}[section]

	\crefname{IntroductionClaim}{proposition}{propositions}
\newtheorem{lemma}{Lemma}[section]

\newtheorem{theorem}{Theorem}[section]
\newtheorem*{theorem*}{Theorem}
\newtheorem{corollary}{Corollary}[section]
\theoremstyle{definition}
\newtheorem{example}{Example}

\newtheorem{definition}{Definition}[section]
\theoremstyle{remark}
\newtheorem{remark}{Remark}[section]

\newcounter{tmp}
\newcounter{constantCounter}
\newcommand{\cste}{{C_{\arabic{constantCounter}}}}
\newcommand{\ncste}{\stepcounter{constantCounter}\cste}
\expandafter\let\expandafter\oldproof\csname\string\proof\endcsname
\let\oldendproof\endproof
\renewenvironment{proof}[1][\proofname]{%
  \oldproof[\setcounter{constantCounter}{0}{#1}.]%
}{\setcounter{constantCounter}{0}\oldendproof}
	
\title[Desingularization of a steady vortex pair in the lake equation]{Desingularization of a steady vortex pair\\in the lake equation}
\date{November 16, 2017}
\subjclass[2010]%
{76C05 (35Q35,35C15,58E30)}

\author{Justin Dekeyser}
\address[Justin Dekeyser]{Université~catholique~de~Louvain,
Département~de~Mathématique,\newline
2~Chemin~du~Cyclotron, 1348~Louvain-La-Neuve, \textsc{Belgium}}
\email[Justin Dekeyser]{Justin.Dekeyser@uclouvain.be}

\newcommand{\reals}{\mathbb{R}}
\newcommand{\plane}{\reals^2}
\newcommand{\integers}{\mathbb{N}}
\DeclareMathOperator{\du}{d}
\newcommand{\dif}{{\,\du}}
\newcommand{\scalarproduct}[3]{({#1}|{#2})_{#3}}
\newcommand{\norm}[2]{\Vert{#1}\Vert_{#2}}

\newcommand{\domain}{\Omega}
\newcommand{\depth}{\textrm{b}}
\newcommand{\closure}[1]{\overline{#1}}

\newcommand{\curve}{\mathcal{C}}
\newcommand{\nbislands}{m}
\newcommand{\capacity}{\textrm{cap}}
\newcommand{\lebesguemeasure}{\dif\textrm{m}}
\newcommand{\mumeasure}{\dif\mu}
\newcommand{\gradient}{\nabla}
\newcommand{\flipgradient}{\gradient^\perp}
\newcommand{\largefunctionspace}{\mathcal{H}}
\newcommand{\functionspace}{\largefunctionspace_0}
\newcommand{\smoothfunctions}{C^1_c}
\newcommand{\divergence}{\textrm{div}}
\newcommand{\curl}{\textrm{curl}}
\newcommand{\circulation}{\textrm{c}}
\newcommand{\staticflow}{\psi}
\newcommand{\positivepart}[1]{\left(#1\right)_+}
\newcommand{\negativepart}[1]{\left(#1\right)_-}
\newcommand{\kernel}{\textrm{Ker}}

\newcommand{\vortex}{\zeta}
\newcommand{\vortexoperator}{\mathfrak{K}}
\newcommand{\rectifycirculation}{\mathfrak{H}}
\newcommand{\boundary}{\partial}
\newcommand{\includesin}{\hookrightarrow}
\newcommand{\laplacian}{-\Delta}
\newcommand{\greenlaplace}{\textrm{g}}
\newcommand{\rectifykernel}{\textrm{R}}
\newcommand{\regulargreenlaplace}{\textrm{H}}
\newcommand{\distance}{\textrm{d}}
\newcommand{\diameter}{\textrm{diam}}
\newcommand{\energy}{\textrm{E}}

\newcommand{\flow}{\Psi}
\newcommand{\rearrangement}{\textrm{Rearg}}

\newcommand{\truncation}[1]{\left[#1\right]}

\newcommand{\symmetrizearoundpoint}[2]{{#1}\sharp{#2}}
\newcommand{\indicator}[1]{\chi_{#1}}
\newcommand{\vortexstrength}{\mathcal{S}}
\newcommand{\threshold}{\varsigma}
\newcommand{\positivefirstorderflow}[1]{\mathcal{T}^+_{#1}}
\newcommand{\negativefirstorderflow}[1]{\mathcal{T}^-_{#1}}
\newcommand{\error}{\Theta}
\newcommand{\externalflowcoefficient}{\lambda}

\begin{document}

\begin{abstract}
We construct a family of steady solutions
of the lake model perturbed by some small Coriolis force, that converge
to a singular vortex pair. The desingularized solutions are
obtained by maximization of the kinetic energy over a class
of rearrangements of sign changing functions. The precise
localization of the asymptotic singular vortex pair
is proved to depend on the depth function and the Coriolis parameter,
and it is independent on the geometry of the lake domain.
We apply our result to construct a singular rotating vortex pair
in a rotation invariant lake.
\end{abstract}
\maketitle


\section*{Introduction}
\begingroup
\setcounter{tmp}{\value{theorem}}
\setcounter{theorem}{0} 
\renewcommand\thetheorem{\Alph{theorem}}

\subsection*{Statement of the problem}
The lake equations arise from the incompressible 3D Euler equations in a regime where the typical velocity magnitude
is small in comparison to the magnitude of gravity waves (small Froude number regime $\textrm{Fr}\ll 1$, see also~\cite{CamassaHolmLevermore}).
Mathematically, a lake can be modeled as a planar open set $\domain\subseteq\plane$ together with a positive depth function $\depth$.
The \emph{velocity field} $v:\reals\times\domain\to\plane$ and the \emph{pressure field} $p:\reals\times\domain\to\reals$
are governed by the system of equations
	\[ \begin{cases}
		\divergence\big( \depth v \big) = 0
			& \text{on}\ \reals\times\domain , \\
		\partial_tv + (v\cdot\gradient)v + fv^\perp = -\gradient p
			& \text{on}\ \reals\times\domain , \\
		\depth v\cdot\hat{\eta} = \nu
			& \text{on}\ \reals\times\boundary\domain .
	\end{cases} \]
Here $\hat{\eta}:\boundary\domain\to\plane$ is the unit outward normal to $\boundary\domain$, $(v_1,v_2)^\perp=(-v_2,v_1)$,
$f:\domain\to\reals$ is a Coriolis parameter, and $\nu:\boundary\domain\to\reals$ is a penetration condition.
When the depth function is constant, $\depth\equiv 1$, the system reduces to the incompressible 2D Euler equations.
The velocity field $v$
in the lake model
may be understood as the horizontal velocity of a column water, whose total mass may vary according to the depth of the lake~\cite{CamassaHolmLevermore}.
Global well-posedness of the lake equations has been studied by
Levermore~\&~Oliver~\&~Titi~\cite{Titi}, Lacave~\&~Nguyen~\&~Pausader~\cite{LacavePausaderNguyen}, Munteanu~\cite{Munteanu},
Huang~\&~Chaocheng~\cite{Huang}.

The problem we study in this article is the asymptotic behavior of steady velocity fields that are constructed as maximizers
of the kinetic energy
	\[ E(v_\epsilon) = \frac{1}{2}\int_\domain|v_\epsilon|^2\mumeasure ,\quad \mumeasure=\depth\lebesguemeasure ,\]
where $\lebesguemeasure$ is the Lebesgue measure in the plane $\plane$; in the regime where the vorticity in the lake
vanishes: $\mumeasure\big(\{\text{vortex exists}\}\big)\to 0$ as $\epsilon\to 0$. Before formulating more precisely this
problem, let us give some more motivation.

Let first consider the simplest case where the Coriolis force is null: $f=0$; and the topography does not vary: $\depth\equiv 1$.
Introducing the vorticity $\omega=\curl(v)$ (and up to prescribe circulation conditions on $\boundary\domain$), one may think of $v$ as
$v=\curl^{-1}(\omega)$, which allows one to think off the energy as a function of $\omega$ only:
	\[ E(v) \equiv E(\omega) .\]
Moreover, applying $\curl$ on the evolution equation (second equation in the lake system), we get the transport equation
	\[ \partial_t\omega + \scalarproduct{v}{\gradient\omega} = 0 .\]
These observations were the starting point for
Arnold~\cite{Arnold} and Benjamin~\cite{Benjamin} to suggest that any steady
solution of the 2D Euler equations should be a critical point of the energy under the following constraint:
if $\omega$ solves the steady transport equation: $\scalarproduct{\curl^{-1}(\omega)}{\gradient\omega}{\plane}=0$; then
for every function $D:\reals\to\text{Diff}_{\lebesguemeasure}(\domain)$ valued in the set of $\lebesguemeasure$-measure preserving diffeomorphisms
of $\domain$ with $D(0)=\text{Id}_{\domain}$, the energy functional
	\[ E : \reals\to\reals : E(t) := E\big( \omega\circ D_t^{-1} \big) \]
has a critical point at $t=0$. On the other hand, any vortex constructed as
$\omega_t = \omega\circ D_t^{-1}$
satisfies the infinite number of constraints:
	\[ \lebesguemeasure\big( \{\omega_t\geq \lambda \} \big) = \lebesguemeasure\big( \omega\geq \lambda \} \big) ,\quad \forall\lambda\in\reals .\]
We say that such a field $\omega_t$ is a $\lebesguemeasure$-rearrangement of $\omega$.
The notion of rearrangement is a relaxed condition in comparison to the identity $\omega_t=\omega\circ D_t^{-1}$, since the
regularity of the transformation $D$ becomes out of concern.
From this motivation, the following energy maximization principle has been extensively investigated by
Arnold~\cite{Arnold}, Benjamin~\cite{Benjamin} and Burton~\citelist{\cite{BurtonRearrangementOfFunctions}\cite{BurtonSteadyConfiguration}\cite{BurtonVariationalProblems}}:
\textit{Find a maximizer of the energy $E$ restricted on the set of all $\lebesguemeasure$-rearrangements of a given profile function $\omega^\star$.}
The asymptotic of maximizers as $\lebesguemeasure\big(\{\omega_\epsilon\neq 0\}\big)=\epsilon^2,\epsilon\to 0$,
has been studied by
Turkington~\citelist{\cite{TurkingtonSteady1}\cite{TurkingtonSteady2}\cite{TurkingtonEvolution}}, Turkington~\&~Friedman~\cite{TurkingtonFriedman},
and Elcrat~\&~Miller~\cite{ElcratMiller} when $\omega_\epsilon$ is non sign changing:
$\omega_\epsilon\geq 0$. The authors have shown that the family of maximizers $\{\omega_\epsilon:\epsilon>0\}$ tends to a Dirac mass as $\epsilon\to 0$,
and the concentration point is chosen according to the geometry of $\boundary\domain$
(see also~\citelist{\cite{Lin1}\cite{Lin2}\cite{TurkingtonEvolution}\cite{MarchioroPulvirenti}} for similar results
for the time-dependent problem, and \citelist{\cite{VanSchaftingenSmets}\cite{VanSchaftingenValeriola}} for similar results based on
other approaches).

The situation appears to be different in the lake model, when $\depth$ is not constant.
In fact, Richardson~\cite{Richardson} has shown that the motion of a singular non sign changing vortex
is led by the depth function $\depth$, independently (at leading order) of the geometry of $\boundary\domain$.
More precisely, Richardson's formula states that the motion of the center of mass $Z_\epsilon$ of a non sign changing potential vortex $\vortex_\epsilon$
with $\mumeasure\big(\{\vortex_\epsilon>0\}\big)=\epsilon^2$,
follows the evolution law
	\[ \frac{d}{dt}Z_\epsilon(t) = \frac{1}{2\pi}\flipgradient\depth(Z_\epsilon(t))\ \log\frac{1}{\epsilon}\vortexstrength_\epsilon + \mathcal{O}(\vortexstrength_\epsilon) ,\]
with $\displaystyle\vortexstrength_\epsilon=\int_\domain\vortex_\epsilon\mumeasure$.
Note that an application of $\depth^{-1}\curl$ on the lake evolution equation yields the new transport equation
	\[ \partial_t\vortex + \scalarproduct{v}{\gradient\vortex} = 0 ,\]
where $\vortex=\depth^{-1}\big(\curl(v)-f\big)$ is the \emph{potential vortex field} associated with the velocity $v$, where the external contribution $f$ has been removed.
Following the lines of thought of Arnold and Benjamin, we expect a variational principle of the following form:
\textit{Steady solutions with prescribed constraints on the $\mumeasure$-measure of their super level sets, can be constructed by
maximization of the energy.}

\subsection*{Formal statement and results}
A measurable function $\vortex:\domain\to\reals$ may be decomposed in a positive part $\positivepart{\vortex}$
and a negative part $\negativepart{\vortex}$, both measurable.
Let us fix a distribution function
	\[ D : \reals^+ \to [0,\mumeasure(\domain)] \]
normalized as
	\[ \int_{\reals^+}D(t)\dif t = 1 ,\]
and such that there exists $p>1$ with
	\[ \int_{\reals^+}t^pD(t)\dif t < +\infty . \]
We also fix $\tau\in[0,1]$.
We say that a family of measurable functions $\big\{\vortex_\epsilon:\epsilon>0\big\}$ satisfies the constraint~\eqref{distributionconstraint}
if, for all $\epsilon>0$, we have
	\begin{equation}\label{distributionconstraint}\tag{D}
		\mumeasure\big( \{\positivepart{\vortex_\epsilon}\geq\lambda \} \big) = \frac{\epsilon^2}{\delta}D\bigg( \frac{\epsilon^2\lambda}{\delta\tau}\,\log\frac{1}{\epsilon} \bigg),
		\qquad
		\mumeasure\big( \{\negativepart{\vortex_\epsilon}\geq\lambda \} \big) = \frac{\epsilon^2}{\delta}D\bigg( \frac{\epsilon^2\lambda}{\delta(1-\tau)}\,\log\frac{1}{\epsilon} \bigg),
	\end{equation}
where $\delta = \sup\limits_{\lambda>0}D(\lambda)$,
so that we always have
$\mumeasure\big( \{\positivepart{\vortex_\epsilon}> 0\} \big) \leq \epsilon^2$ and
$\mumeasure\big( \{\negativepart{\vortex_\epsilon}> 0\} \big) \leq \epsilon^2$.
It is also easy to check that $\vortex_\epsilon\in L^p(\domain,\mumeasure)$ and
	\[ \int_\domain\positivepart{\vortex_\epsilon}\mumeasure = \frac{\tau}{\log\frac{1}{\epsilon}} ,
		\quad\text{and}\quad
	\int_\domain\negativepart{\vortex_\epsilon}\mumeasure = \frac{1-\tau}{\log\frac{1}{\epsilon}} .\]
The above growth factor $\log\frac{1}{\epsilon}$ is motivated by Richardson's formula.
The energy $\energy_\epsilon$ we maximize is given by
	\[ \energy(\vortex) = \frac{1}{2}\int_\domain \big( \vortexoperator\vortex + \flow \big)\vortex\mumeasure ,\]
where $flow$ is the flow that induces the Coriolis vorticity $f$: thus $\curl\big(\depth^{-1}\flipgradient\flow\big)=\depth f$;
and $\vortexoperator\vortex$ is the flow that induces the ``internal vorticity'':
$\curl\big(\depth^{-1}\flipgradient\vortexoperator\vortex\big)=\depth\vortex$.
We prove the following theorem:
\begin{theorem}\label{theoremSteadyIntroduction}
Assume that $\domain\subseteq\plane$ is a (simply connected) bounded domain of class $C^1$ and assume that $\depth\in C(\closure{\domain})\cap W^{1,\infty}_{\text{loc}}(\domain)$
with $\inf_\domain\depth>0$, or $\depth=\phi^\alpha$ for some $\alpha>0$ and $\phi$ a regularization of the distance at the boundary $\boundary\domain$.
Let $\flow\in C(\closure{\domain})\cap W^{1,\infty}_{\text{loc}}(\domain)$.
There exists $\lambda_0>0$ depending only on $\tau,\flow$ and $\depth$ such that, for all $\lambda\in(0,\lambda_0)$,
there exists a family $\big\{ \vortex_\epsilon\in L^p(\domain,\mumeasure):\epsilon>0 \big\}$ of steady solutions for the lake equations,
obtained by maximization of the energy $\energy$ over the class of all functions
that satisfy constraint~\eqref{distributionconstraint}.

Furthermore, if $\tau>0$, the only possible accumulation points of $\{\positivepart{\vortex_\epsilon}:\epsilon>0\}$ as $\epsilon\to 0$, in the sense of
convergence of probability measures on $\domain$,  are Dirac masses $\delta_{x^\star}$ with
	\[ \frac{\tau\depth(x^\star)}{4\pi} + \lambda\flow(x^\star) = \sup_{\domain}\bigg\{ \frac{\tau\depth}{4\pi} + \lambda\flow \bigg\} ,\]
and if $\tau<1$, the only possible accumulation points of $\{\negativepart{\vortex_\epsilon}:\epsilon>0\}$ as $\epsilon\to 0$, in the sense of
convergence of probability measures on $\domain$,  are Dirac masses $\delta_{x_\star}$ with
	\[ \frac{(1-\tau)\depth(x_\star)}{4\pi} - \lambda\flow(x_\star) = \sup_{\domain}\bigg\{ \frac{(1-\tau)\depth}{4\pi} - \lambda\flow \bigg\} .\]
\end{theorem}
We briefly mention that \cref{theoremSteadyIntroduction} will actually be proved for non necessarily simply connected domains,
provided coherent circulation conditions are prescribed. Also,
\cref{theoremSteadyIntroduction} will be proved under slightly more general regularity assumptions on
$\domain$ and $\depth$, and covers two situations encountered in the literature:
the case of a non vanishing depth function $\depth$~\citelist{\cite{Huang}\cite{Titi}}, and the case of a degenerated depth function that vanishes as a polynomial of some
regularized distance at the boundary~\citelist{\cite{BreshMetivier}\cite{LacavePausaderNguyen}}.
Although \cref{theoremSteadyIntroduction} is stated for regular domains, we will be in position to deal with irregular domains as well.

In the case of a non perturbed lake ($\flow=0$),
our energy maximizers look like singular vortex pairs as $\epsilon\to 0$, both located near a point of maximal depth $\depth$.
If we add some small perturbation $\lambda\flow$, the two parts of the pairs separate from each others, according to $\flow$.
The condition that the \emph{vortex strength}
	\[ \int_\domain\vortex_\epsilon\mumeasure = (2\tau-1)\frac{1}{\log\frac{1}{\epsilon}} \]
vanishes as $\big( \log\frac{1}{\epsilon} \big)^{-1}$ is expected from Richardson's law. In the lake equations,
this scale length is the typical scale length
for the dynamic of the vortex to become physically relevant.

Taking advantage of radial symmetries, \cref{theoremSteadyIntroduction} may be used to prove the following time-dependent result:
\begin{theorem}\label{theoremRotatingIntroduction}
Assume that $\domain=B(0,1)$ and $\depth\in C(\closure{\domain})\cap W^{1,\infty}_{\text{loc}}(\domain)$, $b(x)=b(|x|^2)$, is a radial function
with $\inf_\domain\depth>0$, or $\depth=\phi^\alpha$ for some $\alpha>0$ and $\phi$ a regularization of the distance at the boundary $\boundary B(0,1)$.
There exists $\nu_0>0$ depending only on $\tau\in[0,1]$ such that, for all $\nu\in(0,\nu_0)$,
there exists a family of curves $\big\{ \vortex_\epsilon\in C(\reals,L^p(\domain,\mumeasure)):\epsilon>0 \big\}$
solving the time-dependent lake equations with the additional property that
for each $t\in\reals$; and the vortex $\vortex_\epsilon(t)$ at time $t$ is obtained from $\vortex_\epsilon(0)$ by a rotation of clockwise angle $\nu t$,
and satisfies constraint~\eqref{distributionconstraint}.

Furthermore, if $\tau>0$, the only possible accumulation points of $\{\positivepart{\vortex_\epsilon(0)}:\epsilon>0\}$ as $\epsilon\to 0$, in the sense of
convergence of probability measures on $\domain$,  are Dirac masses $\delta_{x^\star}$ with
	\[ \frac{\tau\depth(x^\star)}{4\pi} + \frac{\nu}{2}\int^{|x^\star|^2}_0\depth(s)\dif s
		= \sup_{z\in\domain}\bigg\{ \frac{\tau\depth(z)}{4\pi} +\frac{\nu}{2}\int^{|z|^2}_0\depth(s)\dif s \bigg\} ,\]
and if $\tau<1$, the only possible accumulation points of $\{\negativepart{\vortex_\epsilon}:\epsilon>0\}$ as $\epsilon\to 0$, in the sense of
convergence of probability measures on $\domain$,  are Dirac masses $\delta_{x_\star}$ with
	\[ \frac{(1-\tau)\depth(x_\star)}{4\pi} - \frac{\nu}{2}\int^{|x^\star|^2}_0\depth(s)\dif s
		= \sup_{z\in\domain}\bigg\{ \frac{(1-\tau)\depth(z)}{4\pi} - \frac{\nu}{2}\int^{|z|^2}_0\depth(s)\dif s \bigg\} .\]
\end{theorem}

\subsection*{Method and insights of the proof}
The condition $\divergence\big( \depth v\big)=0$ together with $\curl(v)=\depth\vortex$ motivates to construct $v$ as
$v=\depth^{-1}\flipgradient\psi$ with $\psi$ such that
	\[ -\divergence\Big( \depth^{-1}\gradient\psi \Big) = \depth\vortex .\]
This is an elliptic equation which may be solved on $\domain$, provided circulation conditions are prescribed.
The function framework for this equation is discussed in the \textit{first section}.
We prove that the solution $\psi$ constructed above
depends on $\vortex$ through an operator $\vortexoperator$. \textit{In appendix}, we prove an integral expansion
	\[ \vortexoperator\vortex(x) = \int_\domain\greenlaplace(x,y)\vortex(y)\mumeasure(y) + \int_\domain F(x,y)\vortex(y)\mumeasure(y) ,\]
where $\greenlaplace$ is the Green's function for $\laplacian$ in $\domain$ with Dirichlet boundary conditions,
and $F:\domain\times\domain\to\reals$ is a bounded and continuous function.
Such Green's function expansion is not known for degenerate elliptic equations (see~\citelist{\cite{GFR4}\cite{GFR1}\cite{GFR2}} for existence of Green's function and
related estimates in the case where $\depth$ is non degenerated). We prove this expansion by regularity theory~\cite{GilbardTrudinger},
under suitable assumptions on both $\domain$ and $\depth$.
We cover the case of non regular domains $\domain$ with $\depth\in C(\closure{\domain})\cap W^{1,\infty}_{\text{loc}}(\domain)$
satisfying $\inf_\domain\depth>0$, or $\depth$ satisfying the ad hoc condition
	\[ \inf_{\substack{q\in[1,+\infty)\\\beta\in[0,2)}}\limsup_{\distance(x,\boundary\domain)}\bigg\{ \frac{1}{(\delta_q(x))^\beta} \ \frac{|\gradient\depth(x)|^2}{\depth(x)} \bigg\}
		< +\infty ;\]
or in other words: there exists $\beta\in [0,2)$, $q\geq 1$ and a neighborhood $\mathcal{U}$ of $\boundary\domain$ such that,
for all $x\in\mathcal{U}$:
	\[ \frac{1}{(\delta_q(x))^\beta} \ \frac{|\gradient\depth(x)|^2}{\depth(x)} \leq C \]
for some $C>0$. Here $\delta_q$ is the mean distance at the boundary of $\domain$ of order $q$, as defined in~\cite{Balinsky};
and it is well defined event for non regular domains.
In comparison with~\cite{Munteanu}, we do not rely on a Muckenhoupt condition.
\textit{In a second section}, we recall some preliminaries in rearrangements theory,
and we use the work of Burton~\citelist{\cite{BurtonGlobal}\cite{BurtonRearrangementOfFunctions}\cite{BurtonSteadyConfiguration}\cite{BurtonVariationalProblems}}
to construct maximizers of the energy that are also solutions of the lake equations.
\textit{In section~3}, we exploit our integral representation to study asymptotic properties of energy maximizers as the vortex profile vanishes.
The strategies of proof is to compare the maximal energy with the energy produces by very symmetric competitors, as one would do
for the Euler equations. However, since the relevant measure is $\mumeasure(x)=\depth(x)\dif x$, we have to face some
technical difficulties. We manage to prove a generalization of \cref{theoremSteadyIntroduction}.
The proof of \cref{theoremRotatingIntroduction} is done in \textit{section~4}.

\endgroup
\setcounter{theorem}{\thetmp}

\section{The stream function}

\subsection{Lakes}

By a \emph{lake} we mean a couple $(\domain,\depth)$ where $\domain\subseteq\plane$ is an bounded open set and $\depth\in C(\closure{\domain})$
is a non negative continuous function. We ask the following additional conditions:
\begin{enumerate}
	\item $\domain = \closure{\domain_0}\setminus\bigcup\limits_{i=0}^\nbislands\curve_i$ for a bounded connected open set $\domain_0\subseteq\plane$
	and disjoint compact connected subsets $\curve_i\subset\closure{\domain_0}$, with the convention that $\curve_0=\boundary\domain_0$;
	\item for all $i=0,\dots,\nbislands$, we have either $\capacity(\curve_i)>0$, or $\depth$ admits a Dini continuous extension by $0$ on $\domain\cup\curve_i$;
	\item for all compact subset $K\subset\domain$, we have $\inf_K\depth>0$.
\end{enumerate}
In the above definition, each $\boundary\curve_i$ represents a shore.
We allow $\curve_i$ to be the boundary of a closed connected compact set (in which case we represent an island),
$\curve_i$ a curve, but we also allow $\curve_i$ to be reduced to a point shore, provided the depth function $\depth$ is sufficiently regular in a neighborhood of $\curve_i$.

\subsection{Basic functional framework}

In this section we introduce a functional framework for the stream function.
We write $W^{1,2}(\domain)$ the standard Sobolev space of square integrable functions on $\domain$ whose distributional derivatives exist as square integrable functions
on $\domain$, endowed with its usual scalar product
	\[ \scalarproduct{\phi}{\psi}{W^{1,2}} = \int_\domain\phi\psi\lebesguemeasure + \int_\domain\scalarproduct{\gradient\phi}{\gradient\psi}{\plane}\lebesguemeasure .\]
Here $\lebesguemeasure$ denotes the Lebesgue measure on $\domain$.
We denote by $\largefunctionspace$ the set of those functions $\phi\in W^{1,2}(\domain)$ such that $\depth^{-1}|\gradient\phi|^2\in L^1(\domain,\lebesguemeasure)$,
endowed with the scalar product
	\[ \scalarproduct{\phi}{\psi}{\largefunctionspace} = \scalarproduct{\phi}{\psi}{W^{1,2}}
		+ \int_\Omega\scalarproduct{\gradient\phi}{\gradient\psi}{\plane}\ \frac{\lebesguemeasure}{\depth} .\]
Since $\depth$ is assumed to be positive on compact subsets of $\domain$, we have the set inclusion $\smoothfunctions(\domain)\subseteq\largefunctionspace$.
We denote by $\functionspace=\closure{\smoothfunctions(\domain)}$ the closure of $\smoothfunctions(\domain)$ in $\largefunctionspace$.
\begin{proposition}\label{equivalentscalarproductonfunctionspace}
The quadratic form
	\[ (\phi,\psi)\in\functionspace \mapsto \scalarproduct{\phi}{\psi}{\functionspace} = \int_\domain\scalarproduct{\gradient\phi}{\gradient\psi}{\plane}\ \frac{\lebesguemeasure}{\depth} \]
defines a scalar product whose induced norm is equivalent to the norm induced by $\largefunctionspace$.
\end{proposition}
\begin{proof}
The fact that it defines a scalar product is straightforward.
For the equivalence between the norms, one first observe that
	\[ \norm{\phi}{\functionspace}^2 \leq \norm{\phi}{\largefunctionspace}^2 \]
for all $\phi\in\functionspace$. For a converse inequality, one may write
	\[ \norm{\phi}{\functionspace}^2 \geq \frac{1}{2}\norm{\phi}{\functionspace}^2 + \frac{1}{2\big(\sup_\domain\depth\big)}\int_\domain|\gradient\phi|^2\lebesguemeasure .\]
Since $\domain$ is a bounded set, we may apply the standard Poincaré's inequality to obtain some constant $\ncste>0$ such that
	\[ \cste \int_\domain|\gradient\phi|^2\lebesguemeasure \geq \norm{\phi}{W^{1,2}}^2 ,\]
and thus we have
	\[ \norm{\phi}{\functionspace}^2 \geq \ncste\big( \norm{\phi}{\functionspace}^2 + \norm{\phi}{W^{1,2}}^2 \big) = \cste\norm{\phi}{\largefunctionspace}^2 . \qedhere\]
\end{proof}

\subsection{Circulations}

In this section we show how we are going to solve a problem of the form
	\begin{equation}\tag{C}\label{circulationproblem}
		\begin{cases}
			-\divergence\Big( \depth^{-1}\gradient\phi \Big) = 0 ,\\
			\displaystyle \oint_{\boundary\curve_i}\scalarproduct{\depth^{-1}\flipgradient\phi}{\tau_i}{\plane} = \circulation_i .
		\end{cases} 
	\end{equation}
Here above, the real numbers $\circulation_0,\dots,\circulation_\nbislands\in\reals$ are fixed, and $\tau_i$ denotes the tangent vector field
associated with $\curve_i$ with clockwise orientation for $i=1,\dots,\nbislands$, counterclockwise orientation for $\curve_0$. Each quantity
	\[ \oint_{\boundary\curve_i}\scalarproduct{\depth^{-1}\flipgradient\phi}{\tau_i}{\plane} \]
represents the \emph{circulation of the velocity field} $\depth^{-1}\flipgradient\phi$ along $\boundary\curve_i$.
The first equation
	\[ \curl\Big( \depth^{-1}\flipgradient\phi \Big) = -\divergence\Big( \depth^{-1}\gradient\phi \Big) = 0 \]
means that the curl of the vector field $\depth^{-1}\flipgradient\phi$ is null, that is, no vortex is produced by this velocity field.
Observe that from the circulation conditions we obtain
	\begin{align*}
	\sum^\nbislands_{i=0}\oint_{\boundary\curve_i}\scalarproduct{\depth^{-1}\flipgradient\phi}{\tau_i}{\plane}
		&= \sum^{\nbislands}_{i=0}\int_{\boundary\curve_i}\scalarproduct{\depth^{-1}\gradient\phi}{-\eta_i}{\plane}
		\\&= -\int_{\boundary\domain}\scalarproduct{\depth^{-1}\gradient\phi}{\eta}{\plane}
		= \int_\domain -\divergence\Big( \depth^{-1}\gradient\phi \Big) ,
	\end{align*}
where $\eta$ is the unit normal outward vector on $\domain$, and $\eta_i$ its restriction on $\boundary\curve_i$.
For problem~\eqref{circulationproblem}, this leads to the consistency condition
	\[ \sum^\nbislands_{i=0}\circulation_i = \int_\domain\curl\Big( \depth^{-1}\flipgradient\phi \Big)\lebesguemeasure = 0 .\]
Since we work in an a priori rough setting (the normal vector $\tau_i$ may not be defined),
we understand problem~\eqref{circulationproblem} in the weak sense.
Let us assume that, for all $i=0,\dots,\nbislands$, we have constructed a function $\staticflow_i\in\largefunctionspace$
that solves
	\begin{equation}\tag{T}\label{tambourproblem}
		\begin{cases}
			-\divergence\Big( \depth^{-1}\gradient\staticflow_i \Big) = 0 ,\\
			\staticflow_i = \delta_{ij}\quad \text{on}\ \boundary\curve_j,\forall j\in\{0,\dots,\nbislands\} .
		\end{cases}
	\end{equation}
Here $\delta_{ij}$ denotes the Kronecher's symbol. Then we would have
	\[ \circulation_j = \int_{\boundary\curve_j}\scalarproduct{\depth^{-1}\gradient\phi}{-\eta_j}{\plane}
		= -\int_{\boundary\domain}\scalarproduct{\staticflow_j\depth^{-1}\gradient\phi}{\eta}{\plane}
		= -\int_\domain\divergence\Big( \staticflow_j\depth^{-1}\gradient\phi \Big) ,\]
and therefore, if $\phi\in\largefunctionspace$ is a solution of problem~\eqref{circulationproblem}, we have
	\[ \circulation_j = -\int_\domain\scalarproduct{\gradient\staticflow_j}{\gradient\phi}{\plane}\ \frac{\lebesguemeasure}{\depth} .\]
It is therefore natural to look for a solution $\phi\in\largefunctionspace$ of problem~\eqref{circulationproblem}
as a linear combination
	\[ \phi = \sum^\nbislands_{i=0}\alpha_i\staticflow_i ,\]
where the coefficients $\alpha_0,\dots,\alpha_\nbislands\in\reals$ should be chosen in such a way that
	\[ -\circulation_j = \sum^\nbislands_{i=0}\alpha_i\int_\domain\scalarproduct{\gradient\staticflow_j}{\gradient\staticflow_i}{\plane}\ \frac{\lebesguemeasure}{\depth} ,
		\quad\text{for all}\ j\in\{0,\dots,\nbislands\} . \]
Such a family $\staticflow_1,\dots,\staticflow_\nbislands$ would thus be useful to construct weak solutions of the circulation problem.
However, one should take into account that $\domain$ may be irregular, so that the notion of trace needed to define $\staticflow_1,\dots,\staticflow_\nbislands$,
may be hard to define.
In this section, we begin by proving that problem~\eqref{tambourproblem} admits a unique solution in a certain weak sense,
and this solution is bounded. Once we have proved the existence of the functions $\staticflow_1,\dots,\staticflow_\nbislands$,
we focus on the existence of a solution for our weak formulation of problem~\eqref{circulationproblem}.
\begin{proposition}\label{existencetambourproblem}
For all $i\in\{0,\dots,\nbislands\}$, there exists a unique function $\staticflow_i\in\largefunctionspace$ that satisfies
	\[ \int_\domain\scalarproduct{\gradient\staticflow_i}{\gradient\phi}{\plane}\ \frac{\lebesguemeasure}{\depth} = 0 ,\quad\text{for all}\ \phi\in\functionspace ,\]
and such that there exists a sequence $(\varphi_n)_{n\in\integers}$ of functions from $C^1\big(\closure{\domain}\big)$ converging in $\largefunctionspace$
to $\staticflow_i$; with $\varphi_n=\delta_{ij}$ on a neighborhood of $\boundary\curve_j$, for all $j\in\{0,\dots,\nbislands\}$.
Moreover, the solution $\staticflow_i$ satisfies $0\leq \staticflow_i\leq 1$ on $\domain$.
\end{proposition}
This kind of existence result is standard~\cite{Titi}*{Lemma~4} and it is usually proved by using Schrauder's theory. Here we propose an original proof based on
purely functional analysis methods, that may be applied for non regular domains.
\begin{proof}
Let us consider the set
	\[ \Gamma_i = \big\{ \varphi\in C^1\big(\closure{\domain}\big) : \varphi=\delta_{ij}\ \text{on a neighborhood of}\ \boundary\curve_j,
		\ \text{for all}\ j\in\{0,\dots,\nbislands\} \big\} .\]
We are going to work in the quotient space $H=\largefunctionspace/\sim$, where $\sim$ is the equivalence relation
	\[ \phi_1\sim\phi_2 \iff \int_\domain |\gradient(\phi_1-\phi_2)|^2\ \frac{\lebesguemeasure}{\depth} = 0 .\]
The equivalence class of a function $\phi\in\largefunctionspace$ is simply denoted by $[\phi]$.
The quotient $H$ is then endowed with the well-defined scalar product
	\[ \scalarproduct{[\phi_1]}{[\phi_2]}{H} := \int_\domain \scalarproduct{\gradient\phi_1}{\gradient\phi_2}{\plane}\ \frac{\lebesguemeasure}{\depth} .\]
We claim that the set
	\[ \Pi_i = \big\{ [\phi] : \phi\in\closure{\Gamma_i} \big\} \]
is a complete and convex subset of $H$. The set $\Gamma_i$ itself is convex, and so is its closure $\closure{\Gamma_i}$ in $\largefunctionspace$.
Thus $\Pi_i$ is convex as well. For the completeness, the key point is to observe that the equivalence class $[\phi]$ of some $\phi\in\functionspace$
only contains $\phi$. Hence the equivalence class of some $[\phi]$ with $\phi\in\closure{\Gamma_i}$
only contains $\phi$. Hence completeness of $\Pi_i$ is consequence of completeness of $\functionspace$.
According to the convex projection theorem~\cite{Brezis}*{Theorem~5.2}, there exists a unique $\staticflow_i\in\closure{\Gamma_i}$
such that $[\staticflow_i]\in\Pi_i$ is the projection of $[0]\in H$ onto $\Pi_i$. For all $t\in\reals$ and for all $\varphi\in\functionspace$, we have
$\staticflow_i+t\varphi\in\closure{\Gamma_i}$ and in particular, we must have
	\[ \norm{[\staticflow_i+t\varphi]}{H}^2 \geq \norm{[\staticflow_i]}{H}^2 ,\]
which also reads as
	\[ |t|\int_\domain|\gradient\varphi|^2\ \frac{\lebesguemeasure}{\depth}
		+ 2\frac{t}{|t|}\int_\domain\scalarproduct{\gradient\varphi}{\gradient\staticflow_i}{\plane}\ \frac{\lebesguemeasure}{\depth} \geq 0 .\]
Letting $t\to 0$ yields the identity
	\[ \int_\domain\scalarproduct{\gradient\varphi}{\gradient\staticflow_i}{\plane}\ \frac{\lebesguemeasure}{\depth} = 0 ,\]
for all $\varphi\in\functionspace$.

Finally, the inequalities $0\leq \staticflow_i\leq 1$ are proved by an adaptation of
weak maximum principle techniques~\cite{GilbardTrudinger}*{Theorem~8.1}.
Let us illustrate this by proving $\staticflow_i\leq 1$.
The lower bound $\staticflow_i\geq 0$ will be proved using similar arguments.
By construction of $\staticflow_i$, the
positive part $\positivepart{\staticflow_i-1}$ belongs to $\functionspace$.
In particular, we have
	\[ 0 = \int_\domain\scalarproduct{\gradient\staticflow_i}{\gradient\positivepart{\staticflow_i-1}}{\plane}\ \frac{\lebesguemeasure}{\depth}
		= \int_\domain |\gradient\positivepart{\staticflow_i-1}|^2\ \frac{\lebesguemeasure}{\depth} .\]
Thus $\gradient\positivepart{\staticflow_i-1}=0$ on $\domain$. This is possible only if $\positivepart{\staticflow_i-1}=0$ on $\domain$, and thus $\staticflow_i\leq 1$.
\end{proof}
Now we turn to the existence of solution of problem~\eqref{circulationproblem}.
We recall that by a weak solution of problem~\eqref{circulationproblem}, we mean a linear combination
	\[ \phi = \sum^\nbislands_{i=0}\alpha_i\staticflow_i ,\]
with $\alpha_0,\dots,\alpha_\nbislands\in\reals$ such that
	\[ -\circulation_j = \sum^\nbislands_{i=0}\alpha_i\int_\domain\scalarproduct{\gradient\staticflow_j}{\gradient\staticflow_i}{\plane}\ \frac{\lebesguemeasure}{\depth} .\]
We require a first lemma:
\begin{lemma}\label{lineartambourcombinationisconstant}
A linear combination
	\[ \phi = \sum^\nbislands_{i=0}\alpha_i\staticflow_i \]
with $\alpha_0,\dots,\alpha_\nbislands\in\reals$ belongs to $\functionspace$ if, and only if, we have $\alpha_i=0$ for all $i=0,\dots,\nbislands$.
In particular,
	\[ \sum^\nbislands_{i=0}\staticflow_i = 1 .\]
\end{lemma}
This lemma strongly uses the regularity assumptions on $b$ near a point shore.
\begin{proof}
If we have $\alpha_i=0$ for all $i\in\{0,\dots,\nbislands\}$, then the linear combination is null and belongs to $\functionspace$.
From this it is easy to see that $1-\sum^\nbislands_{i=0}\staticflow$ belongs to $\functionspace$ by construction, and
therefore $\sum^\nbislands_{i=0}\staticflow_i=1$.
For the converse, let us assume that $\sum^\nbislands_{i=0}\alpha_i\staticflow_i$ belongs to $\functionspace$.
By construction of each $\staticflow_i$ and by definition of $\functionspace$
there exists a sequence $(\varphi_n)_{n\in\integers}$ of smooth functions $\varphi_n\in C^1\big(\closure{\domain}\big)$
and $\varphi_n=\alpha_i$ on a neighborhood of $\curve_i$; and the sequence $(\varphi_n)_{n\in\integers}$ converges in $\largefunctionspace$
to $0$. Up to a substraction by $\alpha_0$, we have found a sequence $(\varphi_n)_{n\in\integers}$
of $\smoothfunctions(\domain)$ converging in $\largefunctionspace$ to $-\alpha_0$. By completeness of $W^{1,2}_0(\plane)$
we have $\alpha_0=0$ and the convergence also occurs in $W^{1,2}_0(\plane)$.
Up to passing to a subsequence, we may assume that the convergence occurs
quasi-everywhere on $\plane$~\cite{Willem}*{Section~7.2}. Let $i\in\{0,\dots,\nbislands\}$ be such that $\capacity(\curve_i)>0$.
Then we must have $\alpha_i=-\alpha_0$.
Assume that there exists $i\in\{0,\dots,\nbislands\}$ such that $\capacity(\curve_i)=0$. Since $\curve_i$ is connected and has null capacity,
Pólya-Szeg\H{o}'s inequality~\cite{Willem}*{Theorem~8.3.14} implies that $\curve_i=\{p_i\}$.
By assumption, the depth function $\depth$ may be extended by $0$ on $\domain\cup\{p_i\}$, and the extension (still denoted $\depth$)
is Dini continuous at $p_i$. In particular, we have for all $\varphi\in C^1(\plane)$ supported in $\domain$:
	\[ |\varphi(p_i)| = \bigg| \int_\domain \scalarproduct{\gradient\varphi(x)}{\frac{p_i-x}{|p_i-x|^2}}{\plane}\ \frac{\lebesguemeasure(x)}{\depth(x)}\bigg|
		\leq \bigg(\int_\domain\frac{\depth(x)}{|p_i-x|^2}\lebesguemeasure(x)\bigg)^{\frac{1}{2}}\,\bigg(\int_\domain|\gradient\varphi|^2\ \frac{\lebesguemeasure}{\depth}\bigg)^{\frac{1}{2}} . \]
There exists a constant $\ncste(\depth,p_i)>0$ such that for all $\varphi\in C^1(\plane)$ supported in $\domain$, we have
	\[ |\varphi(p_i)|\leq \cste \bigg(\int_\domain|\gradient\varphi|^2\ \frac{\lebesguemeasure}{\depth}\bigg)^{\frac{1}{2}}\leq \cste\norm{\varphi}{\largefunctionspace}  .\]
Applying this inequality on the sequence $(\varphi_n)_{n\in\integers}$ yields $|\alpha_i|=\lim\limits_{n\to+\infty}|\varphi_n(p_i)|=-\alpha_0=0$.
\end{proof}
\begin{theorem}\label{existencecirculationproblem}
The symmetric matrix $\mathcal{A}\in\reals^{(\nbislands+1)\times(\nbislands+1)}$ defined by
	\[ \mathcal{A}_{ij} = \int_\domain\scalarproduct{\gradient\staticflow_i}{\gradient\staticflow_j}{\plane}\ \frac{\lebesguemeasure}{\depth} \]
has kernel spanned by $(1,\dots,1)\in\reals^{\nbislands+1}$ and non negative eigenvalues. In particular, for any choice $\circulation_0,\dots,\circulation_\nbislands\in\reals$
such that
	\[ \sum^\nbislands_{i=0}\circulation_i = 0 ,\]
there exists a unique weak solution $\phi\in\largefunctionspace$ for problem~\eqref{circulationproblem} of the form $\phi = \sum^\nbislands_{i=1}\alpha_i\staticflow_i$.
\end{theorem}
\begin{proof}
According to \cref{lineartambourcombinationisconstant}, we already know that $(1,\dots,1)$ belongs to the kernel $\kernel(\mathcal{A})$ of $\mathcal{A}$.
Let $(\alpha_0,\dots,\alpha_{\nbislands})$ be another element of $\kernel(\mathcal{A})$.
We compute
	\[ \int_\domain \Big| \gradient \sum^\nbislands_{i=0}\alpha_i\staticflow_i \Big|^2\ \frac{\lebesguemeasure}{\depth}
		= (\alpha_0,\dots,\alpha_\nbislands)\mathcal{A} (\alpha_0,\dots,\alpha_\nbislands)^\dagger = 0 ,\]
where the $M^\dagger$ denotes the transpose of $M$. In particular, the linear combination $\sum^\nbislands_{i=0}\alpha_i\staticflow_i$ is a constant function.
According to \cref{lineartambourcombinationisconstant}, we have $\alpha_0=\alpha_i$ for all $i\in\{0,\dots,\nbislands\}$, and thus $\kernel(\mathcal{A})$
is spanned by $(1,\dots,1)$. Now if $\lambda\in\reals$ is an eigenvalue of $\mathcal{A}$, with eigenvector $\alpha=(\alpha_0,\dots,\alpha_\nbislands)$,
the same computation shows that
	\[ 0\leq \int_\domain \Big| \gradient \sum^\nbislands_{i=0}\alpha_i\staticflow_i \Big|^2\ \frac{\lebesguemeasure}{\depth} = \lambda |\alpha|^2 ,\]
and thus $\lambda\geq 0$. Finally, since $\mathcal{A}$ is symmetric, its range is the orthogonal space of its kernel in $\reals^{\nbislands+1}$,
and the restriction of $\mathcal{A}$ to its range defines a linear isomorphism. The conclusion follows.
\end{proof}

\subsection{Construction of the stream function}

In this section, we show how to construct a stream function $\psi\in\largefunctionspace$ that solves
	\[
		\begin{cases}
			-\divergence\Big( \depth^{-1}\gradient\psi \Big) = \omega ,\\
			\displaystyle \oint_{\boundary\curve_i}\scalarproduct{\depth^{-1}\flipgradient\psi}{\tau_i}{\plane} = \circulation_i ,
		\end{cases}
	\]
for given circulations $\circulation_0,\dots,\circulation_\nbislands$ that satisfy
	\[ \sum^\nbislands_{i=0}\circulation_i = \int_\domain\omega\lebesguemeasure .\]
The function $\omega$ is the physical vortex associated to the velocity field $\depth^{-1}\flipgradient\psi$.
As mentioned in the introduction, it will be suitable to work with the potential vortex $\vortex=\depth^{-1}\omega$ instead,
so we change our notations and we solve in $\largefunctionspace$:
	\begin{equation}\tag{P}\label{completeproblem}
		\begin{cases}
			-\divergence\Big( \depth^{-1}\gradient\psi \Big) = \depth\vortex ,\\
			\displaystyle \oint_{\boundary\curve_i}\scalarproduct{\depth^{-1}\flipgradient\psi}{\tau_i}{\plane} = \circulation_i .
		\end{cases}
	\end{equation}
The functional space of interest for the vortex data $\vortex$ is the Lebesgue space $L^p(\domain,\mumeasure)$, where $p>1$
and the measure $\mumeasure(x)=\depth(x)\lebesguemeasure(x)$ is the invariant measure associated to the lake $(\domain,\depth)$.
Observe that since $\depth$ is a bounded function, we automatically have $(\depth\vortex)\in L^p(\domain,\lebesguemeasure)$.

Let us explain in what sense we understand the weak form of the problem~\eqref{completeproblem}.
Assume that $\psi\in\largefunctionspace$ solves the above problem,
and let $\phi\in\smoothfunctions(\domain)$ be a smooth test function.
By the divergence theorem we have formally, if $\psi$ is a solution of problem~\eqref{completeproblem}:
	\[ \int_\domain\phi\vortex\mumeasure = -\int_\domain\phi\divergence\Big(\depth^{-1}\gradient\psi\Big)\lebesguemeasure
		= \int_\domain \scalarproduct{\gradient\psi}{\gradient\phi}{\plane}\ \frac{\lebesguemeasure}{\depth} .\]
Since we have continuous embeddings
	\[ (\largefunctionspace,\norm{\cdot}{\largefunctionspace}) \includesin (W^{1,2}(\domain),\norm{\cdot}{W^{1,2}})
		\includesin (L^{p'}(\domain,\lebesguemeasure),\norm{\cdot}{L^{p'}_{\lebesguemeasure}})
		\includesin (L^{p'}(\domain,\mumeasure),\norm{\cdot}{L^{p'}_{\mumeasure}}) ,\]
a weak formulation of the elliptic equation in problem~\eqref{completeproblem} would read as
	\[ \int_\domain\phi\vortex\mumeasure = \scalarproduct{\psi}{\phi}{\functionspace} ,\quad\text{for all}\ \phi\in\functionspace .\]
Let us now turn to the interpretation of the circulation conditions.
As in the previous section, we compute for all $j\in\{0,\dots,\nbislands\}$, at least formally:
	\[ \circulation_j = -\int_{\curve_j}\scalarproduct{\depth^{-1}\gradient\psi}{-\eta_j}{\plane}
		= -\int_{\boundary\domain}\scalarproduct{\staticflow_j\depth^{-1}\gradient\psi}{-\eta}{\plane}
		= -\int_\domain\divergence\Big( \staticflow_j\depth^{-1}\gradient\psi \Big) ,\]
so that
	\[ -\circulation_j+\int_\domain\staticflow_j\vortex\mumeasure = \int_\domain\scalarproduct{\gradient\staticflow_i}{\gradient\psi}{\plane}\ \frac{\lebesguemeasure}{\depth} . \]
Accordingly, we propose the following weak formulation of problem~\eqref{completeproblem}.
\begin{definition}\label{weakformulationdefinition}
A function $\psi\in\largefunctionspace$ is a weak solution of problem~\eqref{completeproblem}
if $\psi=u+\rectifycirculation(\vortex)$ for $u\in\functionspace$
such that
	\[ \scalarproduct{u}{\phi}{\functionspace} = \int_\domain\phi\vortex\mumeasure,\quad\text{for all}\ \phi\in\functionspace ,\]
and
	\[ \rectifycirculation(\vortex) = \sum^\nbislands_{i=0}\alpha_i\staticflow_i \]
for $\alpha_0,\dots,\alpha_\nbislands\in\reals$ such that $\alpha_0=0$ and for all $j\in\{0,\dots,\nbislands\}$:
	\[ -\circulation_j+\int_\domain\staticflow_j\vortex\mumeasure = \sum^\nbislands_{i=0}\alpha_i\int_\domain\scalarproduct{\gradient\staticflow_i}{\gradient\staticflow_j}{\plane}\ \frac{\lebesguemeasure}{\depth} . \]
\end{definition}
\begin{theorem}\label{existencecompleteproblem}
Given $\vortex\in L^p(\domain,\mumeasure)$ for $p>1$ and prescribed circulations $\circulation_0,\dots,\circulation_\nbislands\in\reals$ with
	\[ \sum^\nbislands_{i=0}\circulation_i = \int_\domain\vortex\mumeasure ,\]
there exists a unique linear combination
	\[ \rectifycirculation(\vortex) = \sum^\nbislands_{i=1}\alpha_i\staticflow_i \in\largefunctionspace ,\]
where $\staticflow_i$ is given by \cref{existencetambourproblem}; and a unique $\vortexoperator(\vortex)\in\functionspace$,
such that the combined flow $\vortexoperator(\vortex)+\rectifycirculation(\vortex)\in\largefunctionspace$
solves problem~\eqref{completeproblem} in the sense of \cref{weakformulationdefinition}.
\end{theorem}
\begin{proof}
Let us consider the functional
	\[ \Phi : \functionspace\to \reals : \phi\mapsto \frac{1}{2}\norm{\phi}{\functionspace}^2 - \int_\domain\phi\vortex\mumeasure .\]
The functional $\Phi$ is strictly convex, hence it admits at most one minimizer. The existence of the minimizer follows from
the lower semi-continuity of the norm and the compact embedding
	\[ (\functionspace,\norm{\cdot}{\functionspace}) \includesin (L^{p'}(\domain,\mumeasure),\norm{\cdot}{L^{p'}_{\mumeasure}}) .\]
We define $\vortexoperator(\vortex)$ as the unique minimizer of $\Phi$. From this variational principle we conclude that
	\[ \scalarproduct{\vortexoperator(\vortex)}{\phi}{\functionspace} = \int_\domain\phi\vortex\mumeasure,\quad\text{for all}\ \phi\in\functionspace ,\]
According to \cref{existencecirculationproblem}, the required coefficients $\alpha_0,\dots,\alpha_\nbislands\in\reals$ are uniquely determined if we impose $\alpha_0=0$.
We thus define
	\[ \rectifycirculation(\vortex) = \sum^\nbislands_{i=0}\alpha_i\staticflow_i .\]
Now by construction of the static flows $\staticflow_i$ for $i=0,\dots,\nbislands$, the function
$\vortexoperator(\vortex)+\rectifycirculation(\vortex)\in\largefunctionspace$ is a weak solution of problem~\eqref{completeproblem}
in the sense of \cref{weakformulationdefinition}. The uniqueness follows from \cref{lineartambourcombinationisconstant}.
\end{proof}
\begin{corollary}[Symmetry]\label{symmetryofoperators}
For all $\vortex_1,\vortex_2\in L^p(\domain,\mumeasure)$, we have
	\[ \int_\domain \vortex_1\vortexoperator(\vortex_2)\mumeasure = \int_\domain\vortex_2\vortexoperator(\vortex_1)\mumeasure \]
and similarly
	\[ \int_\domain \vortex_1\rectifycirculation(\vortex_2)\mumeasure = \int_\domain\vortex_2\rectifycirculation(\vortex_1)\mumeasure . \]
\end{corollary}
\begin{corollary}[Linearity]\label{linearityofoperators}
For all $\vortex_1,\vortex_2\in L^p(\domain,\mumeasure)$ and for all $\alpha\in\reals$, we have
	\[ \vortexoperator(\vortex_1+\alpha\vortex_2) = \vortexoperator(\vortex_1) + \alpha\vortexoperator(\vortex_2) .\]
\end{corollary}
\begin{corollary}[Positivity]\label{positivityofoperators}
For all $\vortex\in L^p(\domain,\mumeasure)$ with $\vortex\geq 0$, we have $\vortexoperator(\vortex) \geq 0$.
\end{corollary}
\begin{corollary}[Boundedness]\label{boundednessofoperators}
There exists constants $C,C'>0$ such that
For all $\vortex\in L^p(\domain,\mumeasure)$ with $\vortex\geq 0$, we have
	\[ \norm{\vortexoperator(\vortex)}{\functionspace} \leq C\,\norm{\vortex}{L^p_{\mumeasure}} ,\]
and
	\[ \norm{\rectifycirculation(\vortex)}{\largefunctionspace}\leq C'\,\norm{\vortex}{L^1_{\mumeasure}} .\]
\end{corollary}
\begin{proposition}\label{weakstrongcontinuity}
If a sequence $(\vortex_n)_{n\in\integers}$ of $L^p(\domain,\mumeasure)$ functions is bounded in $L^p(\domain,\mumeasure)$
and weakly converges to some $\vortex\in L^p(\domain,\mumeasure)$, then the sequences $(\vortexoperator(\vortex_n))_{n\in\integers}$
and $(\rectifycirculation(\vortex_n))_{n\in\integers}$ both converge strongly, respectively to $\vortexoperator(\vortex)$ in $\functionspace$
and $\rectifycirculation(\vortex)$ in $\largefunctionspace$.
\end{proposition}
\begin{proof}
We first prove the claim for the operator $\vortexoperator$.
Since we have a compact embedding $(\functionspace,\norm{\cdot}{\functionspace})\includesin(L^{p'}(\domain,\mumeasure),\norm{\cdot}{L^{p'}_{\mumeasure}})$,
it is sufficient to prove that every accumulation point of the sequence $(\vortexoperator(\vortex_n))_{n\in\integers}$
in the sense of strong topology of $L^{p'}(\domain,\mumeasure)$ equals $\vortexoperator(\vortex)$,
and that the convergence also occurs in $(\functionspace,\norm{\cdot}{\functionspace})$. Let then $u\in L^{p'}(\domain,\mumeasure)$
be an accumulation point of $(\vortexoperator(\vortex_n))_{n\in\integers}$ in $L^{p'}(\domain,\mumeasure)$. There exists a subsequence
$(\vortexoperator(\vortex_{n_k}))_{k\in\integers}$ converging in $L^{p'}(\domain,\mumeasure)$ to $u\in L^{p'}(\domain,\mumeasure)$.
In particular, for all test function $\phi\in\smoothfunctions(\domain)$, we have
	\[ \int_\domain\phi u\mumeasure = \lim_{k\to+\infty}\int_\domain\phi \vortexoperator(\vortex_{n_k})\mumeasure
			= \lim_{k\to+\infty}\int_\domain\vortexoperator(\phi)\vortex_{n_k}\mumeasure = \int_\domain\vortexoperator(\phi)\vortex\mumeasure
			= \int_\domain\phi\vortexoperator(\vortex)\mumeasure. \]
Therefore $u=\vortexoperator(\vortex)$ in $L^p(\domain,\mumeasure)$, that is, almost-everywhere on $\domain$.
On the other hand, we also have
	\[ \norm{\vortexoperator(\vortex)-\vortexoperator(\vortex_{n_k})}{\functionspace}^2
		= \int_\domain\vortex\vortexoperator(\vortex)\mumeasure + \int_\domain\vortex{n_k}\vortexoperator(\vortex_{n_k})\mumeasure
			- 2\int_\domain\vortex_{n_k}\vortexoperator(\vortex) \mumeasure .\]
The right side converges to $0$, and therefore $(\vortexoperator(\vortex_{n_k}))_{k\in\integers}$ converges strongly in $(\functionspace,\norm{\cdot}{\functionspace})$
to the function $\vortexoperator(\vortex)$.

Let us now prove the similar statement for the operator $\rectifycirculation$.
By construction, the coefficients $\alpha_0,\dots,\alpha_\nbislands$ given by \cref{existencecirculationproblem} depends on
the vortex function $\vortex$ only through the perturbed circulations
	\[ -\circulation_j + \int_\domain\staticflow_j\vortex\mumeasure .\]
In particular, weak convergence of the sequence $(\vortex_n)_{n\in\integers}$ implies strong convergence in $\reals$
of each perturbed circulations, and so strong convergence in $\reals$ of each coefficient $\alpha_j$.
Therefore the sequence $(\rectifycirculation(\vortex_n))_{n\in\integers}$ strongly converges in $\largefunctionspace$ to some $u\in\largefunctionspace$
and by uniqueness, we must have the identity $u=\rectifycirculation(\vortex)$.
\end{proof}

\subsection{Continuous lakes}

Let us write $\greenlaplace:\domain\times\domain\to\reals$
the Green's function for the Laplace's operator $\laplacian$ with Dirichlet's boundary conditions on $\boundary\domain$. We extend $\greenlaplace$
by $0$ on $\closure{\domain}\times\closure{\domain}$. The following
lemma ensures that this extension makes sense, at least in the
sense of weak convergence in every $W^{1,q}_0(\domain),q\in(1,2)$:
\begin{lemma}
The function $\Big[y\in\closure{\domain}\mapsto \greenlaplace(\cdot,y)\Big]$
is continuous in the sense of
weak convergence in $W^{1,q}_0(\domain)$, for any $q\in(1,2)$.
\end{lemma}
\begin{proof}
Let $q\in(1,2)$ be fixed. The fact that $\greenlaplace(\cdot,y)$
is uniformly bounded in $W^{1,q}_0(\domain)$ as $y$ varies
in $\closure{\domain}$ follows from standard argument.
Let $(y_n)_{n\in\integers}$ be a sequence
of points in $\domain$ converging to some $y\in\closure{\domain}$.
The sequence $\big(\greenlaplace(\cdot,y_n)\big)_{n\in\integers}$
admits at least one accumulation point $g\in W^{1,q}_0(\domain)$
in the sense of weak convergence in $W^{1,q}_0(\domain)$. We are done if
we prove $g=\greenlaplace(\cdot,y)$, since the weak convergence on
a bounded set defines a metrizable topology.
Let $L:W^{1,q}_0(\domain)\to\reals$
be any non null continuous linear functional,
where the norm of $W^{1,q}_0(\domain)$
is understood as the gradient norm in $L^q(\domain)$.
According to James' representation theorem\cite{Willem}*{Proposition~5.2.3},
there exists a (unique) $u\in W^{1,q}_0(\domain)$ such that
$\norm{u}{W^{1,q}_0}=1$ and
	\[ L(v)
		= \norm{L}{(W^{1,q}_0)^\star}\int_\domain\scalarproduct{\gradient v}{\gradient u}{\plane}\,|\gradient u|^{q-2}
		= \norm{L}{(W^{1,q}_0)^\star}\int_\domain\scalarproduct{\gradient v}{\gradient\frac{|u|^{q-1}}{q-1}}{\plane} .\]
Now it is readily checked that there exists a sequence $(\varphi_n)_{n\in\integers}$ of $C^1_c(\domain)$ functions converging in $W^{1,\frac{q}{q-1}}_0(\domain)$ to $|u|^{q-1}$, so that
	\[ L\big(\greenlaplace(\cdot,y)-g\big)
		= \lim_{n\to+\infty}\norm{L}{(W^{1,q}_0)^\star}\int_\domain\scalarproduct{\gradient\big(\greenlaplace(\cdot,y)-g\big)}{\gradient\varphi_n}{\plane} = 0 .\]
In the last line we have used the definition of the Green's function
together with the definition of the weak limit $g$.
This shows that every element of $\big(W^{1,q}_0(\domain)\big)^\star$
vanishes on $\greenlaplace(\cdot,y)-g$. This forces
$\greenlaplace(\cdot,y)=g$, and therefore the sequence
$\big(\greenlaplace(\cdot,y_n)\big)_{n\in\integers}$ converges
weakly to $\greenlaplace(\cdot,y)$.
\end{proof}
As a matter of fact, we mention that
the uniform boundedness argument for the Green's function that we rely on
in the previous lemma may be proved through at least three different
approaches: through regularity and a duality argument
in the spirit of~\cite{GFR1} where the characterization of the dual of $L^\infty(\domain)$ as a set of finitely additive measures;
from complex analysis through
the Riemann conformal mapping theorem; or from direct measure geometric like
argument~\cite{GFR2}. We do not enter into details.
\begin{definition}\label{continuouslakes}
A lake $(\domain,\depth)$ is said to be \emph{continuous} if
	\begin{enumerate}
		\item the function $\depth\in C(\closure{\domain})$ admits weak derivatives,
		\item there exists $\ell>2$ such that function
			\[ y\in\domain \mapsto \int_\domain g(\cdot,y)^\ell \ \bigg(\frac{|\gradient b|^2}{b}\bigg)^{\frac{\ell}{2}}\ \lebesguemeasure \]
		is bounded on $\closure{\domain}$,
		\item for all $y_\star\in\closure{\domain}$, we have
			\[ \lim_{y\to y_\star}\int_\domain |g(\cdot,y)-g(\cdot,y_\star)|^\ell \ \bigg(\frac{|\gradient b|^2}{b}\bigg)^{\frac{\ell}{2}}\ \lebesguemeasure = 0 . \]
	\end{enumerate}
\end{definition}
Note that the fact that $\depth$ is positive on compact subsets implies that $\depth\in W^{1,\ell}_{\text{loc}}(\domain)$, for some $\ell>2$.
In particular we must have $C^{0,\alpha}(K)$ for all compact subsets $K\Subset\domain$, and for some $\alpha>0$ depending on $K$.
The main theorem we are going to use on continuous lakes is the following:
\begin{theorem}\label{reduciblelake}
Let $(\domain,\depth)$ be a continuous lake.
There exists a bounded measurable function $\rectifykernel:\domain\times\domain\to\reals$
such that, for all $\vortex\in L^p(\domain,\mumeasure)$, $p>1$, we have almost-everywhere
	\[ \vortexoperator\vortex(x) + \rectifycirculation\vortex(x)
		= \depth(x)\int_\domain\greenlaplace(x,y)\vortex(y)\mumeasure(y) + \int_\domain\rectifykernel(x,y)\vortex(y)\mumeasure(y) .\]
Furthermore, we have $\rectifykernel(\cdot,y)\in\functionspace$
for all $y\in\domain$, with for all $\varphi\in\functionspace$:
	\[ \int_\domain\scalarproduct{\gradient\rectifykernel(\cdot,y)}{\gradient\varphi}{\plane}\ \frac{\lebesguemeasure}{\depth}
	= -\int_\domain\scalarproduct{\greenlaplace(\cdot,y)\gradient\depth}{\gradient\varphi}{\plane}\ \frac{\lebesguemeasure}{\depth} .\]
\end{theorem}
The proof of \cref{reduciblelake} is based on standard ideas from regularity theory.
The so-constructed function $\rectifykernel$ is measurable on the product space $\domain\times\domain$, which will allow us
to manipulate it through Fubini's theorem. In fact, we can prove that $\rectifykernel$ is continuous on $\domain\times\domain$,
but this extra assumption will not be used in the text.
The details are done in \cref{appendix}. Below we give examples of continuous lakes $(\domain,\depth)$.
\begin{example}
If $\depth\in W^{1,\infty}(\domain)$ with $\inf_\domain\depth>0$, then $(\domain,\depth)$ is a continuous lake.
Here we rely on the property that $\greenlaplace(\cdot,y)$ converges weakly in $W^{1,q}_0(\domain)$
to $\greenlaplace(\cdot,y_\star)$ as $y$ converges to $y_\star$; then we use Rellich-Kondrashov's theorem to obtain
strong convergence in $L^p(\domain,\lebesguemeasure)$ for all $p\in[1,+\infty)$.
\end{example}
\begin{example}
If $\delta:\closure{\domain}\to\reals$ is a regularization of the distance at the boundary,
then any function $\depth:\domain\to\reals$ that satisfies
	\[ \limsup_{\delta(x)\to 0}\bigg\{ \frac{1}{\delta(x)^\beta}\,\frac{|\gradient\depth(x)|^2}{\depth(x)} \bigg\} < +\infty \]
for some small exponent $\beta\in[0,2)$, yields to a continuous lake $(\domain,\depth)$.
This follows from the Hardy's inequality for regular domains $\domain$.
This family of examples covers the cases of $\depth=\delta^\alpha$, for sufficiently well behaving $\alpha\in C^1(\closure{\domain})$.
Indeed we have
	\[ \frac{|\gradient\depth|^2}{\depth}
		\sim \delta^\alpha|\log\delta|^2|\gradient\alpha|^2 + \delta^{\alpha-2}\alpha|\gradient\delta|^2 .\]
If $\inf_\domain\alpha>0$, then $\alpha-2>-2$ uniformly on $\domain$, and the Hardy's inequality may be applied.
If, for example, $\alpha=\delta^p$ for some $p>1$, then we would have
	\[ \frac{|\gradient\depth|^2}{\depth} \sim \delta^{\delta^p+2(p-1)}|\log\delta|^2 + \delta^{\delta^p-2+p}|\gradient\delta|^2 \]
and again $\delta^p-2+p>-1$ uniformly on $\domain$. Observe that the special case $(\domain,\depth)$ with $\domain=B(0,1)$
and $\depth(x)=1-|x|$, does not satisfy the $\mathcal{A}_2$ Muckenhoupt condition.
\end{example}
\begin{example}\label{examplesingulardomain}
Let us consider $\domain=[-1,1]^2\setminus\{(0,y):y\leq 0\}$ the square of side length $2$ centered on $0$,
without one half inner vertical median. On this set we consider the function
	\[ \depth : \domain\to\reals^+ : \depth(x_1,x_2) = |x_1| + x_2\chi_{\{x_2>0\}} .\]
So the lake $(\domain,\depth)$ has a line shore.
We compute
	\[ \frac{|\gradient\depth(x_1,x_2)|^2}{\depth(x_1,x_2)} \leq \frac{2}{|x_1| + x_2\chi_{\{x_2>0\}}} .\]
By linearity of the integral, we do not lose in generality in splitting $\domain$ into three disjoints parts:
	\begin{align*}
		\domain_1 &:= \big\{(x_1,x_2) \in\domain : |x_1|<1, 0<x_2<1 \big\} ,\\
		\domain_2 &:= \big\{(x_1,x_2) \in\domain : (x_1\geq 1, x_2>0)\ \text{or}\ (x_2\geq 1) \big\} ,\\
		\domain_3 &:= \big\{(x_1,x_2) \in\domain : x_2\leq 0 \big\} .
	\end{align*}
On $\domain_2$ the function $\depth^{-1}|\gradient\depth|^2$ is uniformly bounded, so one may rely on the first example to
control the contribution from $\domain_2$. On $\domain_1$ we always have
	\[ \frac{|\gradient\depth(x_1,x_2)|^2}{\depth(x_1,x_2)} \leq \frac{2}{\sqrt{x_1^2+x_2^2}} ,\]
and therefore it belongs to $L^q(\domain_1)$ for all $q\in[1,2)$. The conclusion on $\domain_1$ then follows as in the previous example.
For the last part, an elementary computation shows that for all $(x_1,x_2)\in\domain_3$, we have
	\[ \frac{|\gradient\depth(x_1,x_2)|^2}{\depth(x_1,x_2)} \leq \frac{2}{|x_1|} \leq \ncste\frac{1}{\distance_p(x_1,x_2)} ,\]
where $\distance_p$ is the mean distance at the boundary $\partial\domain$ of order $p>1$, defined for all $x\in\domain$ by
	\[ ( \distance_p(x) )^{-1} = 2\bigg( \int_{\mathbb{P}_2}\frac{1}{(\rho_\nu(x))^p}\,\dif\sigma(\nu) \bigg)^\frac{1}{p} ,\]
where $\mathbb{P}_2$ denotes the $2$-dimensional projective plane endowed with its Haar measure $\dif\sigma$, and $\rho_\nu(x)$
is the least distance at $x$ to the boundary in the direction $\nu$~\cite{Balinsky}*{Chapter~3}.
One may then use the Hardy's inequality on general domains~~\cite{Balinsky}*{Theorem~3.3.2}
to obtain the conclusion on $\domain_3$.
\end{example}

\section{Maximization of the energy}

Let us consider an energy functional of the form
	\[ \energy : L^p(\domain,\mumeasure)\to\reals : \energy(\vortex) = \frac{1}{2}\int_\domain\vortex(\vortexoperator+\rectifycirculation)(\vortex)\mumeasure
		+ \int_\domain\flow\vortex\mumeasure .\]
The flow $\flow\in\largefunctionspace$ is a function function independent of the vortex $\vortex$.
The assumption $\flow\in\largefunctionspace$ is of physical importance, since it means that the associated velocity $\depth^{-1}\gradient\flow$
brings a finite contribution to the kinetic energy of the system.

\subsection{Rearrangement of functions}

Let $\vortex\in L^p(\domain,\mumeasure)$ be a given function. We recall that by a $\mumeasure$-rearrangement of $\vortex$, we mean
a function $\tilde{\vortex}\in L^p(\domain,\mumeasure)$ such that, for all $\lambda\in\reals$, we have
	\[ \mumeasure\big(\{\vortex\geq \lambda\}\big) = \mumeasure\big(\{\tilde{\vortex}\geq \lambda\}\big) .\]
This defines an equivalence relation on $L^p(\domain,\mumeasure)$. The set of all $\mumeasure$-rearrangements of a given function $\vortex\in L^p(\domain,\mumeasure)$
will be denoted by $\rearrangement(\vortex)$.
Observe that for $\vortex\in L^p(\domain,\mumeasure)$ and $\xi\in\rearrangement(\vortex)$, we have
$\positivepart{\xi}\in\rearrangement\big(\positivepart{\vortex}\big)$ and similarly $\negativepart{\xi}\in\rearrangement\big(\negativepart{\vortex}\big)$.
It then follows from the Cavalieri's principle (see also~\cite{Rakotoson})
that for all $\tilde{\vortex}\in\rearrangement(\vortex)$, we have
	\[ \bigg( \int_\domain|\xi|^q\mumeasure \bigg)^{\frac{1}{q}} = \bigg( \int_\domain|\vortex|^q\mumeasure \bigg)^{\frac{1}{q}} ,\]
for any $q\in[1,+\infty)$. In particular, for $p\in(1,+\infty)$, the set $\rearrangement(\vortex)$ is closed with respect to the strong topology of $L^p(\domain,\mumeasure)$,
and relatively compact with respect to the weak topology.

The following proposition provides a way to construct ``well behaving'' rearrangements.
The proof is standard in symmetrization theory, but we recall it for completeness.
\begin{proposition}\label{symmetrizearoundpoint}
For all $x\in\domain$, there exists a function $\symmetrizearoundpoint{x}{\cdot}:L^1(\domain,\mumeasure)\to L^1(\domain,\mumeasure)$
such that for all positive function $\vortex\in L^1(\domain,\mumeasure)$ we have
	\begin{enumerate}
		\item $\symmetrizearoundpoint{x}{\vortex}\in\rearrangement(\vortex)$;
		\item the superlevel sets of $\symmetrizearoundpoint{x}{\vortex}$ are balls in $\domain$ centered on $x$.
	\end{enumerate}
\end{proposition}
\begin{proof}
For all $\lambda\in\reals$, we define
	\[ r_\lambda = \inf\Big\{ r\geq 0 : \mumeasure\big(\domain\cap \closure{B(x,r)}\big) \leq \mumeasure\big(\{\vortex\geq \lambda\}\big) \Big\} .\]
We have $r_{\lambda_1}\leq r_{\lambda_2}$ as soon as $\lambda_1\geq \lambda_2$.
We then define the function
	\[ \symmetrizearoundpoint{x}{\vortex} : \domain\to\reals^+ :
		\symmetrizearoundpoint{x}{\vortex}(y) = \sup\big\{ \lambda\in\reals : y\in \domain\cap\closure{B(x,r_\lambda)} \big\} .	\]
If then follows from the definition of $(\symmetrizearoundpoint{x}{\vortex})$ that for all $\lambda\in\reals$:
	\[ \{ \symmetrizearoundpoint{x}{\vortex} \geq \lambda \} = \domain \cap \closure{B\Big(x,\inf_{s<\lambda}r_s\Big)} = \domain\cap \closure{B(x,r_\lambda)} .\]
On the other hand, we have
	\[ \mumeasure\big( \domain\cap \closure{B(x,r_\lambda)} \big) = \mumeasure\big(\{\vortex\geq \lambda\}\big) , \]
and therefore $\symmetrizearoundpoint{x}{\vortex}$ is a $\mumeasure$-rearrangement of $\vortex$ whose superlevel sets are balls centered on $x$.
\end{proof}

\subsection{Convexity of the energy}

\begin{proposition}\label{energyconvexity}
The energy functional $\energy$ is strictly convex and for all $\vortex\in L^p(\domain,\mumeasure)$,
the function $(\vortexoperator+\rectifycirculation)(\vortex)+\flow$
belongs to the subgradient of $\energy$ at point $\vortex$.
\end{proposition}
\begin{proof}
The third contribution
	\[ \vortex\in L^p(\domain,\mumeasure)\mapsto \int_\domain\flow\vortex\mumeasure \]
is linear, hence convex. The first contribution
	\[ \vortex\in L^p(\domain,\mumeasure)\mapsto \frac{1}{2}\int_\domain\vortex\vortexoperator(\vortex)\mumeasure = \frac{1}{2}\norm{\vortexoperator(\vortex)}{\functionspace}^2 \]
is strictly convex. It remains to prove that the second contribution
	\[ \vortex\in L^p(\domain,\mumeasure)\mapsto \frac{1}{2}\int_\domain\vortex\rectifycirculation(\vortex)\mumeasure \]
is convex. Let $\vortex_1,\vortex_2\in L^p(\domain,\mumeasure)$ and $t\in[0,1]$. A computation shows that,
if $\vortex=\vortex_1-\vortex_2$, we have
	\[ \int_\domain\vortex\rectifycirculation(\vortex)\mumeasure = \sum^\nbislands_{i=0}\alpha_i\int_\domain\vortex\staticflow_i\mumeasure
		= (\alpha_0,\dots,\alpha_\nbislands)\mathcal{A}(\alpha_0,\dots,\alpha_\nbislands)^\dagger ,\]
where $\mathcal{A}$ is the matrix constructed in \cref{existencecirculationproblem}. According to
\cref{existencecirculationproblem}, the matrix $\mathcal{A}$ has non negative eigenvalue, and therefore the above quantity is non negative as well.
In particular, we have
	\[ \frac{1}{2}\int_\domain(t\vortex_1+(1-t)\vortex_2)\rectifycirculation(t\vortex_1+(1-t)\vortex_2)\mumeasure
		\leq t\frac{1}{2}\int_\domain\vortex_1\rectifycirculation(\vortex_1)\mumeasure + (1-t)\frac{1}{2}\int_\domain\vortex_2\rectifycirculation(\vortex_2)\mumeasure .\]
For the last claim, we have to prove that, for all $\tilde{\vortex}\in L^p(\domain,\mumeasure)$, we have
	\[ \energy(\tilde{\vortex}) - \energy(\vortex) \geq \int_\domain (\tilde{\vortex}-\vortex)\big( \vortexoperator(\vortex)+\rectifycirculation(\vortex)+\flow\big)\mumeasure .\]
By linearity, it is sufficient to check that
	\[ \int_\domain\tilde{\vortex}(\vortexoperator+\rectifycirculation)(\tilde{\vortex})\mumeasure - \int_\domain\vortex(\vortexoperator+\rectifycirculation)(\vortex)\mumeasure
		\geq 2\int_\domain (\tilde{\vortex}-\vortex)(\vortexoperator+\rectifycirculation)(\vortex)\mumeasure .\]
Here we have used the fact that the eigenvalues of $\mathcal{A}$ are non negative.
\end{proof}

\subsection{Energy maximization and  pressure field}

In this section we show that $\energy$ admits a maximizer over $\rearrangement(\vortex)$,
and that such maximizer leads to solution of the steady lake equations.
\begin{proposition}\label{pressurefield}
Let $\vortex\in L^p(\domain,\mumeasure)$, $p>1$. The restriction of the energy $\energy$ on $\rearrangement(\vortex)$
admits at least one maximizer $\tilde{\vortex}$.
If $\tilde{\vortex}u\in L^1_{\text{loc}}(\domain,\plane)$ is defined by
	\[ u=\depth^{-1}\flipgradient\Big(\vortexoperator(\tilde{\vortex})+\rectifycirculation(\tilde{\vortex})+\flow\Big) ,\]
then for all $\phi\in\smoothfunctions(\domain)$, we have
	\[ \int_\domain\scalarproduct{\tilde{\vortex}u}{\gradient\phi}{\plane}\ \mumeasure = 0 .\]
\end{proposition}
The proof strongly relies on an existence theorem due to Burton~\cite{BurtonRearrangementOfFunctions}*{Theorem~A},
which has been proved in the case of sign changing vortices by an adaptation of the bathtub principle~\cite{LiebLoss}.
\begin{proof}
By~\cite{BurtonRearrangementOfFunctions}*{Theorem~A}, there exists a maximizer $\tilde{\vortex}\in\rearrangement(\vortex)$ of $\energy$ over $\rearrangement(\tilde{\vortex})$,
and a non decreasing function $G:\reals\to\reals$ such that
	\[ \tilde{\vortex} = G\Big( \vortexoperator(\tilde{\vortex})+\rectifycirculation(\tilde{\vortex})+\flow \Big) .\]
Let us write for short $\psi=\vortexoperator(\tilde{\vortex})+\rectifycirculation(\tilde{\vortex})+\flow$, and assume that
	\[ \tilde{\vortex}\,\depth^{-1}\flipgradient\psi \in L^1_{\text{loc}}(\domain,\plane) .\]
We define, for all $n\in\integers$:
	\[ \truncation{G}_n : \reals\to\reals : \truncation{G}_n(t) = \min\big\{ n, \max\big\{ G(t), -n\big\} \big\} . \]
For all $n\in\integers$, the function $\truncation{G}_n$ is non decreasing and bounded.
We define similarly the truncation at $k\in\integers$ of the function $\psi$, so that:
	\[ \truncation{\psi}_k\in \largefunctionspace,\qquad\gradient\truncation{\psi}_k= \indicator{\{|\psi|\leq k\}}\gradient\psi .\]
The function
	\[ F_{n,k} : [-n,+\infty)\to\reals : \int_{-k}^t\truncation{G}_n(s)\dif s \]
is Lipschitz continuous.
The function $F_{n,k}\circ \truncation{\psi}_k$ then belongs to $W^{1,2}(\domain)$ with
	\[ \gradient (F_{n,k}\circ \truncation{\psi}_k) = (\truncation{G}_n\circ\truncation{\psi}_k)\ \gradient\truncation{\psi}_k .\]
As a result, it follows from the divergence theorem that for all $\phi\in\smoothfunctions(\domain)$, we have
	\[ \int_\domain\scalarproduct{\flipgradient\truncation{\psi}_k}{\gradient\phi}{\plane}\ (\truncation{G}_n\circ\truncation{\psi}_k)\lebesguemeasure = 0 .\]
For all $n\in\integers$ we have
	\[ \Big| \scalarproduct{\flipgradient\truncation{\psi}_k}{\gradient\phi}{\plane}\ (\truncation{G}_n\circ\truncation{\psi}_k) \Big|
		\leq \Big| \scalarproduct{\flipgradient\psi}{\gradient\phi}{\plane}\ \tilde{\vortex} \Big| ,\]
and the latter belongs to $L^1(\domain,\mu)$ by Hölder's inequality. On the other hand, we have almost-everywhere on $\domain$:
	\[ \lim_{k\to +\infty}\scalarproduct{\flipgradient\truncation{\psi}_k(x)}{\gradient\phi(x)}{\plane}\ \truncation{G}_n(\truncation{\psi}_k(x))
		= \scalarproduct{\flipgradient\psi(x)}{\gradient\phi(x)}{\plane}\ \truncation{G}_n(\psi(x)) ,\]
because $\gradient\psi=0$ almost-everywhere on a set of the form $\{\psi=\alpha\}$.
Since $G$ is monotone, the set of discontinuities of $G$ is at most countable, hence we miss at most a countable union of negligible sets in $\domain$.
According to the Lebesgue's dominated convergence theorem,
we obtain
	\[ \int_\domain\scalarproduct{ \tilde{\vortex}\flipgradient\psi}{\gradient\phi}{\plane}\ \lebesguemeasure = 0 , \]
which is equivalent to the conclusion.
\end{proof}

\section{Asymptotic behavior of maximizers}\label{sectionAsymptotic}

In this section, we turn our attention to the asymptotic behavior of a maximizing family
$\big\{\vortex_\epsilon\in L^p(\domain,\mumeasure) : \epsilon>0\big\}$ with related energies
	\[ E_\epsilon(\vortex) = \frac{1}{2}\int_\domain\vortex\vortexoperator(\vortex)+\vortex\rectifycirculation_\epsilon(\vortex)\mumeasure
		+ \int_\domain\externalflowcoefficient_{\epsilon}\flow\vortex\mumeasure .\]
We assume that the vortex profiles obey the identities
	\[ \tau\vortexstrength_\epsilon = \int_\domain\positivepart{\vortex_\epsilon}\mumeasure ,
		\qquad  (1-\tau)\vortexstrength_\epsilon = \int_\domain\negativepart{\vortex_\epsilon}\mumeasure  \]
for some $\tau\in[0,1]$ and $\vortexstrength_\epsilon>0$.
We make the following ($\mumeasure$-rearrangement invariant) extra assumption:
	\[ \sup_{\epsilon>0}\Bigg\{
		\frac{\norm{\positivepart{\vortex_\epsilon}}{L^p_{\mumeasure}}\epsilon^{2(1-\frac{1}{p})}}{\tau\vortexstrength_\epsilon}
		+ \frac{\norm{\negativepart{\vortex_\epsilon}}{L^p_{\mumeasure}}\epsilon^{2(1-\frac{1}{p})}}{(1-\tau)\vortexstrength_\epsilon}
	\Bigg\} < +\infty .\]
Here above, the convention $\displaystyle\frac{0}{0}=0$ is used. This extra condition is motivated by
the fact that any family that satisfies constraint~\eqref{distributionconstraint} (page~\pageref{distributionconstraint}) as in the introduction
also satisfies the above control condition.

More precisely, if we are given a distribution function
	\[ D : \reals^+ \to [0,\mumeasure(\domain)] \]
normalized as
	\[ \int_{\reals^+}D(t)\dif t = 1 ,\]
and such that there exists $p>1$ with
	\[ \int_{\reals^+}t^pD(t)\dif t < +\infty ,\]
and if we define, for all $\epsilon>0$, a reference profile $\tilde{\vortex}_\epsilon$ such that
	\[ \mumeasure\big( \{\positivepart{\tilde{\vortex}_\epsilon}\geq\lambda \} \big) = \frac{\epsilon^2}{\delta}D\bigg( \frac{\epsilon^2\lambda}{\delta\tau}\,\log\frac{1}{\epsilon} \bigg),
		\qquad
		\mumeasure\big( \{\negativepart{\tilde{\vortex}_\epsilon}\geq\lambda \} \big) = \frac{\epsilon^2}{\delta}D\bigg( \frac{\epsilon^2\lambda}{\delta(1-\tau)}\,\log\frac{1}{\epsilon} \bigg),
	\]
where $\delta = \sup\limits_{\lambda>0}D(\lambda)$, then using the axiom of choice a family of maximizers $\{\vortex_\epsilon:\epsilon>0\}$
as above always exists. In other words, the condition we impose on the $L^p(\domain,\mumeasure)$-norms is satisfied by every family constructed
by some scaling process.

The circulations $\circulation_{\epsilon;0},\dots,\circulation_{\epsilon;\nbislands}\in\reals$
that come into play in the definition of operator $\rectifycirculation_\epsilon$ are assumed to depend on $\epsilon>0$ through
	\[ \circulation_{\epsilon;i}=\circulation_i(2\tau-1)\vortexstrength_\epsilon \]
for $\circulation_0,\dots,\circulation_\nbislands\in\reals$ independent of $\epsilon>0$, and
	\[ \sum^{\nbislands}_{i=0}\circulation_i = 1 .\]
We also assume that we have $\externalflowcoefficient_\epsilon\geq 0$ with
	\[ \lim_{\epsilon\to 0}\frac{\externalflowcoefficient_{\epsilon}}{\vortexstrength_\epsilon\log\frac{1}{\epsilon}} = 1 .\]

On the lake $(\domain,\depth)$ we make the assumption that $(\domain,\depth)$ is continuous, in the sense of \cref{continuouslakes}.
We also make the following assumption that the external flow $\flow$
is continuous on $\closure{\domain}$, and sufficiently small so that:
	\begin{enumerate}
		\item there exists $\threshold>0$ such that, for all sufficiently small $\epsilon>0$, if $\tau>0$:
			\[ \sup_\domain\bigg\{
				\frac{4\pi\externalflowcoefficient_{\epsilon}|\flow|}{(\sup_\domain\depth)\vortexstrength_\epsilon\log\frac{1}{\epsilon}}
			\bigg\} \leq \frac{\tau}{2}-\threshold; \]
		and if $\tau<1$:
			\[ \sup_\domain\bigg\{
				\frac{4\pi\externalflowcoefficient_{\epsilon}|\flow|}{(\sup_\domain\depth)\vortexstrength_\epsilon\log\frac{1}{\epsilon}}
			\bigg\} \leq \frac{(1-\tau)}{2}-\threshold .\]
	\end{enumerate}
If $\tau>0$ (resp.~$\tau<1$), then the above conditions ensure that the function
	\[ \phi : \closure{\domain}\to\reals : \phi(x) = \frac{\tau\depth(x)}{4\pi}+\frac{\externalflowcoefficient_{\epsilon}}{\vortexstrength_\epsilon\log\frac{1}{\epsilon}}\flow(x) \]
(resp.~$\depth/(4\pi)-\phi$) never reaches a maximal value on a shore, that is: on some point $x_\star$ where $\depth(x_\star)=0$.
This qualitative information will be used several times, but we will also require the quantitative estimate
to show that the positive and the negative parts of the vortex remain concentrated.

Finally, we write $\dif\vortex(x)=\frac{\vortex}{\norm{\vortex}{L^1_{\mumeasure}}}\mumeasure$ for all non null function $\vortex\in L^p(\domain,\mumeasure)$,
and $\dif 0=0$.

\subsection{Leading partial flows}

The leading partial flows induced by some $\vortex\in L^p(\domain, \mumeasure)$ are defined
for all $x\in\domain$ by
	\[ \positivefirstorderflow{\vortex}(x) = \frac{\depth(x)}{4\pi}\int_\domain\log\frac{\diameter(\domain)}{|x-y|}\ \positivepart{\vortex}(y)\mumeasure(y)
		+ \externalflowcoefficient_{\epsilon}\flow(x) , \]
and
	\[ \negativefirstorderflow{\vortex}(x) = \frac{\depth(x)}{4\pi}\int_\domain\log\frac{\diameter(\domain)}{|x-y|}\ \negativepart{\vortex}(y)\mumeasure(y)
		- \externalflowcoefficient_{\epsilon}\flow(x) . \]
\begin{proposition}\label{energyupperboundbyfows}
Let
	\[ \greenlaplace(x,y) = \frac{1}{2\pi}\log\frac{\diameter(\domain)}{|x-y|} - \regulargreenlaplace(x,y) \]
be the Green's function for the Laplace's operator $\laplacian$ on $\domain$ with Dirichlet boundary conditions,
and let $\rectifykernel$ be given by \cref{reduciblelake}. Then we have for all $\vortex\in L^p(\domain,\mumeasure)$ and for all $\epsilon>0$:
	\begin{multline*}
		\energy_\epsilon(\vortex) = \tau\vortexstrength_\epsilon\int_\domain\positivefirstorderflow{\vortex}\dif\positivepart{\vortex}
			+ (1-\tau)\vortexstrength_\epsilon\int_\domain\negativefirstorderflow{\vortex}\dif\negativepart{\vortex}
			\\ + \frac{1}{2}\iint_{\domain\times\domain} \big(\rectifykernel(x,y)-b(x)\regulargreenlaplace(x,y)\big)\,\vortex(x)\vortex(y)\mumeasure\otimes\mumeasure(x,y)
			\\ - \iint_{\domain\times\domain}\frac{\depth(x)+\depth(y)}{4\pi}\log\frac{\diameter(\domain)}{|x-y|}
				\ \positivepart{\vortex}(x)\negativepart{\vortex}(y)\mumeasure\otimes\mumeasure(x,y) .
	\end{multline*}
In particular, there exists a constant $C>0$ such that
	\[ \energy_\epsilon(\vortex) \leq \tau\vortexstrength_\epsilon\int_\domain\positivefirstorderflow{\vortex}\dif\positivepart{\vortex}
		+ (1-\tau)\vortexstrength_\epsilon\int_\domain\negativefirstorderflow{\vortex}\dif\negativepart{\vortex}
		+ C\vortexstrength_\epsilon^2 .\]
\end{proposition}
The proof is straightforward from the fact that $(\domain,\depth)$ is a continuous lake, by using \cref{reduciblelake}.
and the properties of $\regulargreenlaplace$, $\rectifykernel$, and the positivity of $\vortexoperator$. We omit it.
\begin{proposition}\label{firstorderflowupperbound}
There exists a constant $C>0$ such that for all $\epsilon>0$ and for all $\vortex\in\rearrangement(\vortex_\epsilon)$:
	\[ \positivefirstorderflow{\vortex}(x) \leq \frac{\tau\depth(x)}{4\pi}\log\frac{1}{\epsilon}\vortexstrength_\epsilon + \externalflowcoefficient_{\epsilon}\flow(x)
		+ C\tau\vortexstrength_\epsilon \]
and
	\[ \negativefirstorderflow{\vortex}(x) \leq \frac{(1-\tau)\depth(x)}{4\pi}\log\frac{1}{\epsilon}\vortexstrength_\epsilon - \externalflowcoefficient_{\epsilon}\flow(x)
		+ C(1-\tau)\vortexstrength_\epsilon .\]
\end{proposition}
\begin{proof}
In both cases, it is sufficient to control the quantity
	\[ \int_\domain\log\frac{\diameter(\domain)}{|x-y|}\ \positivepart{\vortex_\epsilon}\mumeasure \]
by a suitable upper bound. Define
	\[ M = \sup_{\epsilon>0}\Bigg\{
		\frac{\norm{\positivepart{\vortex_\epsilon}}{L^p_{\mumeasure}}\epsilon^{2(1-\frac{1}{p})}}{\tau\vortexstrength_\epsilon}
		+ \frac{\norm{\negativepart{\vortex_\epsilon}}{L^p_{\mumeasure}}\epsilon^{2(1-\frac{1}{p})}}{(1-\tau)\vortexstrength_\epsilon}
	\Bigg\} < +\infty ,\]
so that
	\[ \sup_{\epsilon>0}\Bigg\{
		\frac{\norm{\positivepart{\vortex_\epsilon}}{L^p_{\mumeasure}}\epsilon^{2(1-\frac{1}{p})}}{\vortexstrength_\epsilon}
	\Bigg\} \leq \tau M < +\infty .\]
In the above condition, one may replace $\vortex_\epsilon$ by any of its $\mumeasure$-rearrangement,
because $\mumeasure$-rearrangements preserve every $L^q(\domain,\mumeasure)$-norms, $q\in[1,+\infty]$.
From this we compute, for all $\vortex\in\rearrangement(\vortex_\epsilon)$:
	\begin{align*}
		\int_\domain\log\frac{\epsilon}{|x-y|}\ \positivepart{\vortex}(y)\mumeasure(y)
			&\leq \int_\domain\positivepart{\log\frac{\epsilon}{|x-y|}}\ \positivepart{\vortex}(y)\mumeasure(y)
			\\&\leq \norm{\positivepart{\vortex}}{L^p_{\mumeasure}}\big(\sup_\domain\depth\big)^\frac{1}{p}
				\Bigg( \int_\domain\positivepart{\log\frac{\epsilon}{|x-y|}}^{p'}\lebesguemeasure(y) \Bigg)^\frac{1}{p'}
			\\&\leq \norm{\positivepart{\vortex}}{L^p_{\mumeasure}}\big(\sup_\domain\depth\big)^\frac{1}{p}
				\Bigg( \int_{B(0,\epsilon)}\positivepart{\log\frac{\epsilon}{|y|}}^{p'}\lebesguemeasure(y) \Bigg)^\frac{1}{p}
			\\&\leq \norm{\positivepart{\vortex}}{L^p_{\mumeasure}}\big(\sup_\domain\depth\big)^\frac{1}{p'}\epsilon^{\frac{2}{p'}}
				\Bigg( \int_{B(0,1)}\positivepart{\log\frac{1}{|y|}}^{p'}\lebesguemeasure(y) \Bigg)^\frac{1}{p'} .
	\end{align*}
Using the definition of $M$, there exists a constant $\ncste>0$ such that
	\[ \int_\domain\log\frac{\epsilon}{|x-y|}\ \positivepart{\vortex}(y)\mumeasure(y) \leq \cste\vortexstrength_\epsilon .\]
The conclusion now follows from the definitions of $\positivefirstorderflow{\vortex}$ and $\negativefirstorderflow{\vortex}$.
\end{proof}

\subsection{Main lower bound on the energy}

The following lemma will be used several times in the text.
\begin{lemma}\label{competitorpointsconstruction}
Let $x_1,x_2\in\closure{\domain}$ be such that $\depth(x_1)>0$ and $\depth(x_2)>0$,
and let $\mathcal{U}_1,\mathcal{U}_2\subseteq\domain$ be open sets containing respectively $x_1$ and $x_2$.
There exists approximation families $\{x^1_\epsilon\in\mathcal{U}_1 : \epsilon>0\}$ and $\{x^2_\epsilon\in\mathcal{U}_2:\epsilon>0\}$ such that
	\[ \lim_{\epsilon\to 0}\distance(x_1,x^1_\epsilon) = 0 ,\qquad \lim_{\epsilon\to 0}\distance(x_2,x^2_\epsilon)=0 ,\]
and the following conditions hold for sufficiently small $\epsilon>0$:
	\begin{itemize}
		\item there exists $r_1>0$ such that $\epsilon^2\leq \mumeasure\Big( B(x^1_\epsilon,r_1\epsilon) \Big)$;
		\item there exists $r_2>0$ such that $\epsilon^2\leq \mumeasure\Big( B(x^2_\epsilon,r_2\epsilon) \Big)$;
		\item $B(x^1,r_1\epsilon)\subseteq\mathcal{U}_1$ and $B(x^2,r_2\epsilon)\subseteq\mathcal{U}_2$;
		\item $\distance\Big( B(x^1,r_1\epsilon) , \boundary\domain \Big) \geq \frac{1}{\log\frac{1}{\epsilon}}$
		and $\distance\Big( B(x^2,r_2\epsilon) , \boundary\domain \Big) \geq \frac{1}{\log\frac{1}{\epsilon}}$;
		\item $\distance\Big( B(x^1,r_1\epsilon), B(x^2,r_2\epsilon) \Big) \geq \frac{1}{\log\frac{1}{\epsilon}} $.
	\end{itemize}
\end{lemma}
We omit the proof of \cref{competitorpointsconstruction}, which may be done by geometric arguments.
Note that since we do not claim anything on the regularity of $\boundary\domain$ at points $x_1$ and $x_2$,
the rate of convergence of the families $\{x^1_\epsilon:\epsilon>0\}$ and $\{x^2_\epsilon:\epsilon>0\}$
to their respective limit is not explicitly known.

\begin{proposition}\label{errortermdefinition}
For all $\epsilon>0$ sufficiently small, there exists $\error_\epsilon\leq 0$ with $\lim\limits_{\epsilon\to 0}\error_\epsilon=0$
and such that there holds
	\begin{multline*}
		\vortexstrength_\epsilon^2\log\frac{1}{\epsilon}\Bigg( \tau\sup_\domain\Bigg\{
			\frac{\tau\depth}{4\pi} + \frac{\externalflowcoefficient_{\epsilon}}{\vortexstrength_\epsilon\log\frac{1}{\epsilon}}\flow
		\Bigg\}
		+ (1-\tau)\sup_\domain\Bigg\{
			\frac{(1-\tau)\depth}{4\pi}-\frac{\externalflowcoefficient_{\epsilon}}{\vortexstrength_\epsilon\log\frac{1}{\epsilon}}\flow
		\Bigg\} \Bigg)
		\\ \leq \energy_\epsilon\big( \vortex_\epsilon \big)
			- \error_\epsilon\vortexstrength_\epsilon^2\log\frac{1}{\epsilon}.
	\end{multline*}
\end{proposition}
\begin{proof}
Define the function
	\[ \phi : \closure{\domain}\to\reals : \phi(z) = \frac{\tau\depth(z)}{4\pi} + \flow(z) . \]
The function $\phi$ reaches a maximal value at some point $x_1\in\{\depth>0\}$,
and the function $\frac{\depth}{4\pi}-\phi$ reaches a maximal value at some point $x_2\in\{\depth>0\}$.
There exists $\eta_1,\eta_2>$ such that the sets $\mathcal{U}_1=\{b>\eta_1\}$ and $\mathcal{U}_2=\{b>\eta_2\}$ contain respectively $x_1$ and $x_2$.
Let $\{x^1_\epsilon\in\mathcal{U}_1:\epsilon>0\}$ and $\{x^2_\epsilon\in\mathcal{U}_2:\epsilon>0\}$
be given by \cref{competitorpointsconstruction}, and $r_1,r_2>0$ the associated radii.
For all $\epsilon>0$ sufficiently small, we compare the maximal energy with the energy
produced by symmetrized pairs $\symmetrizearoundpoint{x^1_\epsilon}{\positivepart{\vortex_\epsilon}}$
and $\symmetrizearoundpoint{x^2_\epsilon}{\negativepart{\vortex_\epsilon}}$. For sufficiently small $\epsilon>0$, the function
$\xi_\epsilon=\big( \symmetrizearoundpoint{x^1_\epsilon}{\positivepart{\vortex_\epsilon}} - \symmetrizearoundpoint{x^2_\epsilon}{\negativepart{\vortex_\epsilon}} \big)$
is a $\mumeasure$-rearrangement of $\vortex_\epsilon$, and
	\[ \big\{ \symmetrizearoundpoint{x^1_\epsilon}{\positivepart{\vortex_\epsilon}} > 0 \big\} \subseteq B(x^1_\epsilon, r_1\epsilon) ,\qquad
		\big\{ \symmetrizearoundpoint{x^2_\epsilon}{\negativepart{\vortex_\epsilon}} > 0 \big\} \subseteq B(x^2_\epsilon, r_2\epsilon) . \]
Now we compute the energy $\energy_\epsilon$ produced by the above competitor $\xi_\epsilon$ using \cref{energyupperboundbyfows}:
	\begin{multline*}
		\energy_\epsilon(\xi_\epsilon) = \int_\domain\positivefirstorderflow{\symmetrizearoundpoint{x^1_\epsilon}{\positivepart{\vortex_\epsilon}}}\ \symmetrizearoundpoint{x^1_\epsilon}{\positivepart{\vortex_\epsilon}}\mumeasure
			+ \int_\domain\negativefirstorderflow{\symmetrizearoundpoint{x^2_\epsilon}{\negativepart{\vortex_\epsilon}}}\ \symmetrizearoundpoint{x^2_\epsilon}{\negativepart{\vortex_\epsilon}}\mumeasure
			\\ + \frac{1}{2}\iint_{\domain\times\domain} \big(\rectifykernel(x,y)-b(x)\regulargreenlaplace(x,y)\big)\,\xi_\epsilon(x)\xi_\epsilon(y)\mumeasure\otimes\mumeasure(x,y)
			\\ - \iint_{\domain\times\domain}\frac{\depth(x)+\depth(y)}{4\pi}\log\frac{\diameter(\domain)}{|x-y|}
				\ \symmetrizearoundpoint{x^1_\epsilon}{\positivepart{\vortex_\epsilon}}(x)\symmetrizearoundpoint{x^2_\epsilon}{\negativepart{\vortex_\epsilon}}(y)\mumeasure\otimes\mumeasure(x,y) .
	\end{multline*}
We recall that $\rectifykernel$ is bounded, and  $\regulargreenlaplace$ is symmetric, positive, continuous and satisfies
	\[ \regulargreenlaplace(x,y) \leq \frac{1}{2\pi}\log\frac{\diameter(\domain)}{\distance(y,\boundary\domain)} . \]
Because the diameter of $\{\symmetrizearoundpoint{x^1_\epsilon}{\positivepart{\vortex_\epsilon}}>0\}$ is smaller than $r_1\epsilon$, we have
	\[ \iint_{\domain\times\domain}\frac{\depth(x)}{4\pi}\log\frac{\diameter(\domain)}{|x-y|}\ \symmetrizearoundpoint{x^1_\epsilon}{\positivepart{\vortex_\epsilon}}
		\symmetrizearoundpoint{x^1_\epsilon}{\positivepart{\vortex_\epsilon}} \mumeasure\mumeasure
	\geq \tau\vortexstrength_\epsilon\log\frac{1}{\epsilon}\int_\domain\frac{\depth}{4\pi}\ \symmetrizearoundpoint{x^1_\epsilon}{\positivepart{\vortex_\epsilon}}\mumeasure
		- \ncste\vortexstrength_\epsilon^2 .\]
for some constant $\cste>0$. Similarly, there exists $\ncste>0$ such that
	\[ \iint_{\domain\times\domain}\frac{\depth(x)}{4\pi}\log\frac{\diameter(\domain)}{|x-y|}\ \symmetrizearoundpoint{x^2_\epsilon}{\negativepart{\vortex_\epsilon}}
		\symmetrizearoundpoint{x^2_\epsilon}{\negativepart{\vortex_\epsilon}} \mumeasure\mumeasure
	\geq (1-\tau)\vortexstrength_\epsilon\log\frac{1}{\epsilon}\int_\domain\frac{\depth}{4\pi}\ \symmetrizearoundpoint{x^2_\epsilon}{\negativepart{\vortex_\epsilon}}\mumeasure
		- \ncste\vortexstrength_\epsilon^2 .\]
By construction (\cref{competitorpointsconstruction}), we also have
	\[ \distance\Big( B(x^1_\epsilon,r_1\epsilon) , B(x^2_\epsilon,r_2\epsilon) \Big) \geq \frac{1}{\log\frac{1}{\epsilon}} ,\]
so that there exists $\ncste>0$ with
	\[ \iint_{\domain\times\domain}\frac{\depth(x)}{4\pi}\log\frac{\diameter(\domain)}{|x-y|}\ \symmetrizearoundpoint{x^1_\epsilon}{\positivepart{\vortex_\epsilon}}
		\symmetrizearoundpoint{x^2_\epsilon}{\negativepart{\vortex_\epsilon}} \mumeasure\mumeasure
	\leq \tau(1-\tau)\cste\vortexstrength_\epsilon^2\bigg(1+\log\log\frac{1}{\epsilon}\bigg) .\]
Recalling the conditions (\cref{competitorpointsconstruction}) that
	\[ \distance\Big( B(x^1_\epsilon,r_1\epsilon) , \boundary\domain \Big) \geq \frac{1}{\log\frac{1}{\epsilon}} ,
		\qquad \distance\Big( B(x^2_\epsilon,r_2\epsilon) , \boundary\domain \Big) \geq \frac{1}{\log\frac{1}{\epsilon}} ,\]
the upper bound for $\regulargreenlaplace$ and the boundedness of $\rectifykernel$,
we obtain
	\begin{align*}
		\energy_\epsilon\big( \symmetrizearoundpoint{x^1_\epsilon}{\positivepart{\vortex_\epsilon}} - \symmetrizearoundpoint{x^2_\epsilon}{\negativepart{\vortex_\epsilon}} \big)
			& \geq \vortexstrength_\epsilon\log\frac{1}{\epsilon}\int_\domain\frac{\tau\depth}{4\pi}\ \symmetrizearoundpoint{x^1_\epsilon}{\positivepart{\vortex_\epsilon}}\mumeasure
			+ \vortexstrength_\epsilon\log\frac{1}{\epsilon}\int_\domain\frac{(1-\tau)\depth}{4\pi}\ \symmetrizearoundpoint{x^2_\epsilon}{\negativepart{\vortex_\epsilon}}\mumeasure
			\\&\qquad + \int_\domain \symmetrizearoundpoint{x^1_\epsilon}{\positivepart{\vortex_\epsilon}}\externalflowcoefficient_{\epsilon}\flow\mumeasure
				- \int_\domain \symmetrizearoundpoint{x^2_\epsilon}{\negativepart{\vortex_\epsilon}}\externalflowcoefficient_{\epsilon}\flow\mumeasure
			\\&\quad - \cste\vortexstrength_\epsilon^2\bigg( 1+\log\log\frac{1}{\epsilon}\bigg) .
	\end{align*}
From this we conclude that
	\begin{align*}
		\energy_\epsilon\big( \symmetrizearoundpoint{x^1_\epsilon}{\positivepart{\vortex_\epsilon}} - \symmetrizearoundpoint{x^2_\epsilon}{\negativepart{\vortex_\epsilon}} \big)
			& \geq \tau\vortexstrength_\epsilon^2\log\frac{1}{\epsilon}\inf_{B(x^1_\epsilon,r_1\epsilon)}\Bigg\{
				\frac{\tau\depth}{4\pi} + \frac{\externalflowcoefficient_{\epsilon}}{\vortexstrength_\epsilon\log\frac{1}{\epsilon}}\flow
			\Bigg\}
			\\&\qquad + (1-\tau)\vortexstrength_\epsilon^2\log\frac{1}{\epsilon}\inf_{B(x^2_\epsilon,r_2\epsilon)}\Bigg\{
				\frac{(1-\tau)\depth}{4\pi}-\frac{\externalflowcoefficient_{\epsilon}}{\vortexstrength_\epsilon\log\frac{1}{\epsilon}}\flow
			\Bigg\}
			\\&\quad - \cste\vortexstrength_\epsilon^2\bigg( 1+\log\log\frac{1}{\epsilon}\bigg) .
	\end{align*}
By definition of the family $\{\vortex_\epsilon:\epsilon>0\}$, we have
	\[ \energy_\epsilon(\xi_\epsilon)
		= \energy_\epsilon\big( \symmetrizearoundpoint{x^1_\epsilon}{\positivepart{\vortex_\epsilon}} - \symmetrizearoundpoint{x^2_\epsilon}{\negativepart{\vortex_\epsilon}} \big)
		\leq \energy_\epsilon\big( \positivepart{\vortex_\epsilon} - \negativepart{\vortex_\epsilon} \big)
		= \energy_\epsilon(\vortex_\epsilon) ,\]
and therefore we have
	\begin{multline*}
		\vortexstrength_\epsilon^2\log\frac{1}{\epsilon}\Bigg( \tau\inf_{B(x^1_\epsilon,r_1\epsilon)}\Bigg\{
			\frac{\tau\depth}{4\pi} + \frac{\externalflowcoefficient_{\epsilon}}{\vortexstrength_\epsilon\log\frac{1}{\epsilon}}\flow
		\Bigg\}
		+ (1-\tau)\inf_{B(x^2_\epsilon,r_2\epsilon)}\Bigg\{
			\frac{(1-\tau)\depth}{4\pi}-\frac{\externalflowcoefficient_{\epsilon}}{\vortexstrength_\epsilon\log\frac{1}{\epsilon}}\flow
		\Bigg\} \Bigg)
		\\ \leq \energy_\epsilon\big( \vortex_\epsilon \big)
			+ \cste\Big(1+\log\log\frac{1}{\epsilon}\Big).
	\end{multline*}
We define the error term
	\begin{multline*}
		\error_\epsilon = \tau\Bigg(\inf_{B(x^1_\epsilon,r_1\epsilon)}\Bigg\{
			\frac{\tau\depth}{4\pi} + \frac{\externalflowcoefficient_{\epsilon}}{\vortexstrength_\epsilon\log\frac{1}{\epsilon}}\flow
			\Bigg\} - \sup_\domain\Bigg\{
				\frac{\tau\depth}{4\pi} + \frac{\externalflowcoefficient_{\epsilon}}{\vortexstrength_\epsilon\log\frac{1}{\epsilon}}\flow
			\Bigg\} \Bigg)
			\\ + (1-\tau)\Bigg( \inf_{B(x^2_\epsilon,r_2\epsilon)}\Bigg\{
				\frac{(1-\tau)\depth}{4\pi}-\frac{\externalflowcoefficient_{\epsilon}}{\vortexstrength_\epsilon\log\frac{1}{\epsilon}}\flow
			\Bigg\} - \sup_\domain\Bigg\{
				\frac{(1-\tau)\depth}{4\pi}-\frac{\externalflowcoefficient_{\epsilon}}{\vortexstrength_\epsilon\log\frac{1}{\epsilon}}\flow
			\Bigg\} \Bigg)
			\\ - \cste\frac{1+\log\log\frac{1}{\epsilon}}{\log\frac{1}{\epsilon}} ,
	\end{multline*}
so that $\error_\epsilon\leq 0$ and
	\begin{multline*}
		\vortexstrength_\epsilon^2\log\frac{1}{\epsilon}\Bigg( \tau\sup_\domain\Bigg\{
			\frac{\tau\depth}{4\pi} + \frac{\externalflowcoefficient_{\epsilon}}{\vortexstrength_\epsilon\log\frac{1}{\epsilon}}\flow
		\Bigg\}
		+ (1-\tau)\sup_\domain\Bigg\{
			\frac{(1-\tau)\depth}{4\pi}-\frac{\externalflowcoefficient_{\epsilon}}{\vortexstrength_\epsilon\log\frac{1}{\epsilon}}\flow
		\Bigg\} \Bigg)
		\\ \leq \energy_\epsilon\big( \vortex_\epsilon \big)
			- \error_\epsilon\vortexstrength_\epsilon^2\log\frac{1}{\epsilon}.
	\end{multline*}
Finally, the uniform continuity of $\depth$ and $\flow$ on the compact set $\closure{\domain}$, shows that we have $\lim\limits_{\epsilon\to 0}\error_\epsilon=0$.
\end{proof}

\subsection{Truncation of the vortex core}

In this section we show that the main part of the vortex core is located in an area
of the domain $\domain$ where the leading partial flows are large.
The truncation process was already used by Turkington in the context of
time-dependent solutions of the Euler equations~\cite{TurkingtonEvolution} (see also~\citelist{\cite{TurkingtonSteady1}\cite{TurkingtonSteady2}}).

\begin{corollary}\label{vortextruncation}
Let $\{\error_\epsilon\in\reals:\epsilon>0\}$ be a family of real numbers satisfying the claim of \cref{errortermdefinition}.
There exists a constant $C>0$ such that for all $\kappa>0$, we have
	\[ \tau\vortexstrength_\epsilon\int_{\domain\setminus D^\kappa_\epsilon}\dif\positivepart{\vortex_\epsilon} \leq
		\frac{1}{\kappa}\frac{4\pi}{\sup_\domain\depth}\vortexstrength_\epsilon\Bigg( \error_\epsilon + \frac{C}{\log\frac{1}{\epsilon}} \Bigg) . \]
with
	\[ D^\kappa_\epsilon = \Bigg\{ x\in\domain : \positivefirstorderflow{\vortex_\epsilon}(x)\geq \int_\domain\positivefirstorderflow{\vortex_\epsilon}\dif\positivepart{\vortex_\epsilon}
		- \kappa\frac{\sup_\domain\depth}{4\pi}\log\frac{1}{\epsilon}\vortexstrength_\epsilon \Bigg\} ; \]
and similarly
	\[ (1-\tau)\vortexstrength_\epsilon\int_{\domain\setminus U^\kappa_\epsilon}\dif\negativepart{\vortex_\epsilon} \leq
		\frac{1}{\kappa}\frac{4\pi}{\sup_\domain\depth}\vortexstrength_\epsilon\Bigg( \error_\epsilon + \frac{C}{\log\frac{1}{\epsilon}} \Bigg) . \]
with
	\[ U^\kappa_\epsilon = \Bigg\{ x\in\domain : \negativefirstorderflow{\vortex_\epsilon}(x)\geq \int_\domain\negativefirstorderflow{\vortex_\epsilon}\dif\negativepart{\vortex_\epsilon}
		- \kappa\frac{\sup_\domain\depth}{4\pi}\log\frac{1}{\epsilon}\vortexstrength_\epsilon \Bigg\} . \]
\end{corollary}
\begin{proof}
It is sufficient to prove the claim for $\positivefirstorderflow{\vortex_\epsilon}$ and $\tau>0$.
By assumption, we have for all sufficiently small $\epsilon>0$:
	\begin{multline*}
		\vortexstrength_\epsilon^2\log\frac{1}{\epsilon}\Bigg( \tau\sup_\domain\Bigg\{
			\frac{\tau\depth}{4\pi} + \frac{\externalflowcoefficient_{\epsilon}}{\vortexstrength_\epsilon\log\frac{1}{\epsilon}}\flow
		\Bigg\}
		+ (1-\tau)\sup_\domain\Bigg\{
			\frac{(1-\tau)\depth}{4\pi}-\frac{\externalflowcoefficient_{\epsilon}}{\vortexstrength_\epsilon\log\frac{1}{\epsilon}}\flow
		\Bigg\} \Bigg)
		\\ \leq \energy_\epsilon\big( \vortex_\epsilon \big)
			- \error_\epsilon\vortexstrength_\epsilon^2\log\frac{1}{\epsilon}.
	\end{multline*}
According to \cref{energyupperboundbyfows}, there exists $\ncste>0$ such that
	\[ \energy_\epsilon\big( \vortex_\epsilon \big)
		\leq \tau\vortexstrength_\epsilon\int_\domain\positivefirstorderflow{\vortex_\epsilon}\dif\positivepart{\vortex_\epsilon}
			+ (1-\tau)\vortexstrength_\epsilon\int_\domain\negativefirstorderflow{\vortex_\epsilon}\dif\negativepart{\vortex_\epsilon}
		+ \cste\vortexstrength_\epsilon^2 .\]
From this we conclude the inequality
	\begin{multline*}
		\vortexstrength_\epsilon^2\log\frac{1}{\epsilon}\Bigg( \tau\sup_\domain\Bigg\{
			\frac{\tau\depth}{4\pi} + \frac{\externalflowcoefficient_{\epsilon}}{\vortexstrength_\epsilon\log\frac{1}{\epsilon}}\flow
		\Bigg\}
		+ (1-\tau)\sup_\domain\Bigg\{
			\frac{(1-\tau)\depth}{4\pi}-\frac{\externalflowcoefficient_{\epsilon}}{\vortexstrength_\epsilon\log\frac{1}{\epsilon}}\flow
		\Bigg\} \Bigg)
		\\ \leq \tau\vortexstrength_\epsilon\int_\domain \positivefirstorderflow{\vortex_\epsilon} \dif\positivepart{\vortex_\epsilon}
			+ (1-\tau)\vortexstrength_\epsilon\int_\domain \negativefirstorderflow{\vortex_\epsilon} \dif\negativepart{\vortex_\epsilon}
			- \error_\epsilon\vortexstrength_\epsilon^2\log\frac{1}{\epsilon} + \cste\vortexstrength_\epsilon^2.
	\end{multline*}
On the other hand, an application of \cref{firstorderflowupperbound} for $\negativefirstorderflow{\vortex_\epsilon}$ yields
	\[ \vortexstrength_\epsilon^2\log\frac{1}{\epsilon} \tau\sup_\domain\Bigg\{
			\frac{\tau\depth}{4\pi} + \frac{\externalflowcoefficient_{\epsilon}}{\vortexstrength_\epsilon\log\frac{1}{\epsilon}}\flow
		\Bigg\}
		\leq \tau\vortexstrength_\epsilon\int_\domain \positivefirstorderflow{\vortex_\epsilon} \dif\positivepart{\vortex_\epsilon}
			- \error_\epsilon\vortexstrength_\epsilon^2\log\frac{1}{\epsilon} + \ncste\vortexstrength_\epsilon^2 .\]
An application of \cref{firstorderflowupperbound} for $\positivefirstorderflow{\vortex_\epsilon}$ yields in turn
	\[ \positivefirstorderflow{\vortex_\epsilon} \leq \vortexstrength_\epsilon\log\frac{1}{\epsilon}\sup_\domain\Bigg\{
		\frac{\tau\depth}{4\pi} + \frac{\externalflowcoefficient_{\epsilon}}{\vortexstrength_\epsilon\log\frac{1}{\epsilon}}\flow
	\Bigg\}
		+ \ncste\vortexstrength_\epsilon . \]
On $\domain$ we thus have
	\[ \tau\vortexstrength_\epsilon\positivefirstorderflow{\vortex_\epsilon}-\tau\vortexstrength_\epsilon\int_\domain\positivefirstorderflow{\vortex_\epsilon}\dif\positivepart{\vortex_\epsilon} 
		\leq \error_\epsilon\vortexstrength_\epsilon^2\log\frac{1}{\epsilon} + \ncste\vortexstrength_\epsilon^2 . \]
Now we compute
	\[ \tau\vortexstrength_\epsilon\int_{\domain\setminus D^\kappa_\epsilon}\Bigg(\int_\domain\positivefirstorderflow{\vortex_\epsilon}\dif\positivepart{\vortex_\epsilon}- \positivefirstorderflow{\vortex_\epsilon}\Bigg)
		\dif\positivepart{\vortex_\epsilon}
	= \tau\vortexstrength_\epsilon\int_{D^\kappa_\epsilon}\Bigg( \positivefirstorderflow{\vortex_\epsilon}-\int_\domain\positivefirstorderflow{\vortex_\epsilon}\dif\positivepart{\vortex_\epsilon}\Bigg)\dif\positivepart{\vortex_\epsilon} .\]
Using the definition of $D^\kappa_\epsilon$ we obtain
	\[ \tau\vortexstrength_\epsilon\int_{\domain\setminus D^\kappa_\epsilon}\dif\positivepart{\vortex_\epsilon}
	\leq \frac{1}{\kappa}\frac{4\pi}{\sup_\domain\depth}\,\vortexstrength_\epsilon\Bigg( \error_\epsilon + \frac{\cste}{\log\frac{1}{\epsilon}} \Bigg) .  \qedhere\]
\end{proof}

\subsection{Concentration of the truncated vortices}

In this section we prove that most of the vortex is located in two small balls: one containing
the positive part of the pair, and the other the negative part of the pair.

\begin{theorem}\label{concentration}
Let $\{\error_\epsilon:\epsilon>0\}$ be a family of real numbers as in \cref{errortermdefinition},
and let
	\[ D^\kappa_\epsilon = \Bigg\{ x\in\domain : \positivefirstorderflow{\vortex_\epsilon}(x)\geq \int_\domain\positivefirstorderflow{\vortex_\epsilon}\dif\positivepart{\vortex_\epsilon}
		- \kappa\frac{\sup_\domain\depth}{4\pi}\log\frac{1}{\epsilon}\vortexstrength_\epsilon \Bigg\} \]
and
	\[ U^\kappa_\epsilon = \Bigg\{ x\in\domain : \negativefirstorderflow{\vortex_\epsilon}(x)\geq \int_\domain\negativefirstorderflow{\vortex_\epsilon}\dif\negativepart{\vortex_\epsilon}
		- \kappa\frac{\sup_\domain\depth}{4\pi}\log\frac{1}{\epsilon}\vortexstrength_\epsilon \Bigg\} .\]
There exists $\sigma,\kappa>0$ such that for all sufficiently small $\epsilon>0$,
	\[ \diameter\big( D^\kappa_\epsilon \big) \leq \epsilon^\sigma \]
and
	\[ \diameter\big( U^\kappa_\epsilon \big) \leq \epsilon^\sigma .\]
\end{theorem}
\begin{proof}
It is sufficient to prove the claim for $\positivepart{\vortex_\epsilon}$ and $\tau>0$.
Recall that we always assume that there exists $\threshold>0$ such that, for all sufficiently small $\epsilon>0$:
	\[ \sup_\domain\bigg\{
		\frac{4\pi\externalflowcoefficient_{\epsilon}|\flow|}{(\sup_\domain\depth)\vortexstrength_\epsilon\log\frac{1}{\epsilon}}
	\bigg\} \leq \frac{\tau}{2}-\threshold .\]
In particular, there exists $s,\delta,\eta\in(0,1)$ closed to $1$ and $\kappa>0$ closed to $0$ such that for all sufficiently small $\epsilon>0$:
	\[ (1-\delta)\tau + (1-\eta)\frac{(1-\tau)^2}{\tau} + \kappa + \sup_\domain\bigg\{
		\frac{4\pi\externalflowcoefficient_{\epsilon}|\flow|}{(\sup_\domain\depth)\vortexstrength_\epsilon\log\frac{1}{\epsilon}}
	\bigg\} < \frac{s\tau}{2} .\]
Let $\{\error_\epsilon\in\reals:\epsilon>0\}$ be a family of real numbers as in \cref{errortermdefinition}.
Let $C>0$ and $D^+_\epsilon$ be given as in \cref{vortextruncation}.
First pick points $x^\star,y^\star\in\domain$ such that
	\[ \frac{\tau\depth(x^\star)}{4\pi}>\frac{\delta+1}{2}\sup_\domain\frac{\tau\depth}{4\pi} \]
and $\depth(y_\star)>0$, $y^\star\neq x^\star$ and
	\[ \frac{(1-\tau)\depth(y^\star)}{4\pi} - \frac{\externalflowcoefficient_{\epsilon}}{\vortexstrength_\epsilon\log\frac{1}{\epsilon}}\flow(y^\star)
		> \frac{\eta+1}{2}\sup_\domain\bigg\{
			\frac{(1-\tau)\depth}{4\pi} -\frac{\externalflowcoefficient_{\epsilon}}{\vortexstrength_\epsilon\log\frac{1}{\epsilon}}\flow
		\bigg\} .\]
Such points exist by continuity. Let us compute the energy produced by the competitor
$\tilde{\vortex}_\epsilon=\symmetrizearoundpoint{x^\star}{\positivepart{\vortex_\epsilon}} - \symmetrizearoundpoint{y^\star}{\negativepart{\vortex_\epsilon}}$.
For sufficiently small $\epsilon>0$, the function $\tilde{\vortex}$ belongs to $\rearrangement(\vortex_\epsilon)$ because $y^\star\neq x^\star$.
Thus for sufficiently small $\epsilon>0$ we have
	\[ \energy_\epsilon\big( \symmetrizearoundpoint{x^\star}{\positivepart{\vortex_\epsilon}} - \symmetrizearoundpoint{y^\star}{\negativepart{\vortex_\epsilon}} \big)
		\leq \energy_\epsilon(\vortex_\epsilon) .\]
According to \cref{energyupperboundbyfows} and \cref{firstorderflowupperbound}, we then have
	\[ \energy_\epsilon(\vortex_\epsilon) \leq \tau\vortexstrength_\epsilon\int_\domain\positivefirstorderflow{\vortex_\epsilon}\dif\positivepart{\vortex_\epsilon}
		+ (1-\tau)\log\frac{1}{\epsilon}\vortexstrength_\epsilon^2\sup_\domain\bigg\{
			\frac{(1-\tau)\depth}{4\pi} - \frac{\externalflowcoefficient_{\epsilon}}{\vortexstrength_\epsilon\log\frac{1}{\epsilon}}\flow
		\bigg\}
		+ \ncste\vortexstrength_\epsilon^2 .\]
On the other hand,
because $x^\star\neq y^\star$ with $\distance(x^\star,\boundary\domain)>0$ and $\distance(y^\star,\boundary\domain)>0$, there exists $\ncste>0$
such that for all sufficiently small $\epsilon>0$, we have
	\begin{multline*}
		\tau\vortexstrength_\epsilon^2\log\frac{1}{\epsilon}\bigg(
			\frac{\delta\tau(\sup_\domain\depth)}{4\pi} - \sup_\domain\bigg\{
				\frac{\externalflowcoefficient_{\epsilon}|\flow|}{\vortexstrength_\epsilon\log\frac{1}{\epsilon}}
			\bigg\} \bigg)
		+ (1-\tau)\vortexstrength_\epsilon^2\log\frac{1}{\epsilon}\ \eta\sup_\domain\bigg\{
			\frac{(1-\tau)\depth}{4\pi} - \frac{\externalflowcoefficient_{\epsilon}\flow}{\vortexstrength_\epsilon\log\frac{1}{\epsilon}}
		\bigg\}
		\\ \leq \energy_\epsilon\big( \symmetrizearoundpoint{x^\star}{\positivepart{\vortex_\epsilon}} - \symmetrizearoundpoint{y^\star}{\negativepart{\vortex_\epsilon}} \big)
			+ \cste\vortexstrength_\epsilon^2 .
	\end{multline*}
From this we conclude that, for all sufficiently small $\epsilon>0$:
	\begin{multline*}	\tau\vortexstrength_\epsilon\int_\domain\positivefirstorderflow{\vortex_\epsilon}\dif\positivepart{\vortex_\epsilon}
		\geq \tau\vortexstrength_\epsilon^2\log\frac{1}{\epsilon}\bigg(
			\frac{\delta\tau(\sup_\domain\depth)}{4\pi} - \sup_\domain\bigg\{
				\frac{\externalflowcoefficient_{\epsilon}}{\vortexstrength_\epsilon\log\frac{1}{\epsilon}}|\flow|
			\bigg\} \bigg)
		\\+ (\eta-1)(1-\tau)\vortexstrength_\epsilon^2\log\frac{1}{\epsilon}\sup_\domain\bigg\{
				\frac{(1-\tau)\depth}{4\pi} - \frac{\externalflowcoefficient_{\epsilon}}{\vortexstrength_\epsilon\log\frac{1}{\epsilon}}|\flow|
			\bigg\} - \ncste\vortexstrength_\epsilon^2
	\end{multline*}
Recalling the definition of $D^\kappa_\epsilon$ in \cref{vortextruncation}, we have for all $x\in D^\kappa_\epsilon$:
	\begin{multline*}
		\tau\vortexstrength_\epsilon\positivefirstorderflow{\vortex_\epsilon}(x)
		\geq \tau\vortexstrength_\epsilon^2\log\frac{1}{\epsilon}\bigg(
			\frac{\delta\tau(\sup_\domain\depth)}{4\pi} - \sup_\domain\bigg\{
				\frac{\externalflowcoefficient_{\epsilon}}{\vortexstrength_\epsilon\log\frac{1}{\epsilon}}|\flow|
			\bigg\} \bigg)
		\\+ (\eta-1)(1-\tau)\vortexstrength_\epsilon^2\log\frac{1}{\epsilon}\sup_\domain\bigg\{
				\frac{(1-\tau)\depth}{4\pi} - \frac{\externalflowcoefficient_{\epsilon}}{\vortexstrength_\epsilon\log\frac{1}{\epsilon}}|\flow|
			\bigg\}
		\\ - \ncste\vortexstrength_\epsilon^2
		- \tau\vortexstrength_\epsilon^2\log\frac{1}{\epsilon}\ \kappa\frac{\sup_\domain\depth}{4\pi} .
	\end{multline*}
Now by definition of $\positivefirstorderflow{\vortex_\epsilon}$, we have
	\[ \positivefirstorderflow{\vortex_\epsilon}(x) \leq \frac{\sup_\domain\depth}{4\pi}\int_\domain\log\frac{\diameter(\domain)}{|x-y|}\ \positivepart{\vortex_\epsilon}(y)\mumeasure(y)
		+ \sup_\domain\big\{\externalflowcoefficient_{\epsilon}|\flow|\big\} , \]
and therefore, for all $x\in D^\kappa_\epsilon$ and sufficiently small $\epsilon>0$:
	\begin{multline*}
		\tau^2\vortexstrength_\epsilon^2\frac{\sup_\domain\depth}{4\pi}\int_\domain\log\frac{\epsilon}{|x-y|}\ \dif\positivepart{\vortex_\epsilon}(y)
		\geq \tau\vortexstrength_\epsilon^2\log\frac{1}{\epsilon}\bigg(
			\frac{(\delta-1)\tau(\sup_\domain\depth)}{4\pi} - \sup_\domain\bigg\{
				\frac{2\externalflowcoefficient_{\epsilon}|\flow|}{\vortexstrength_\epsilon\log\frac{1}{\epsilon}}
			\bigg\} \bigg)
		\\+ (\eta-1)(1-\tau)\vortexstrength_\epsilon^2\log\frac{1}{\epsilon}\sup_\domain\bigg\{
				\frac{(1-\tau)\depth}{4\pi} - \frac{\externalflowcoefficient_{\epsilon}}{\vortexstrength_\epsilon\log\frac{1}{\epsilon}}|\flow|
			\bigg\}
		\\ - \ncste\vortexstrength_\epsilon^2
		- \tau\vortexstrength_\epsilon^2\log\frac{1}{\epsilon}\ \kappa\frac{\sup_\domain\depth}{4\pi} .
	\end{multline*}
Regrouping terms, we obtain
	\begin{multline*}
		\tau^2\vortexstrength_\epsilon^2\int_\domain\log\frac{\epsilon}{|x-y|}\ \dif\positivepart{\vortex_\epsilon}(y)
		\geq - \ncste\vortexstrength_\epsilon^2
		- \vortexstrength_\epsilon^2\log\frac{1}{\epsilon}\bigg(
			(1-\delta)\tau^2 + (1-\eta)(1-\tau)^2 + \kappa\tau
		\bigg)
		\\- \tau\vortexstrength_\epsilon^2\log\frac{1}{\epsilon}
			\sup_\domain\bigg\{\frac{8\pi\externalflowcoefficient_{\epsilon}}{(\sup_\domain\depth)\vortexstrength_\epsilon\log\frac{1}{\epsilon}}|\flow|\bigg\}
	\end{multline*}
Since $\tau>0$, we obtain in particular
	\begin{multline*}
		\tau\vortexstrength_\epsilon\int_\domain\log\frac{\epsilon}{|x-y|}\ \dif\positivepart{\vortex_\epsilon}(y)
		\geq - \ncste\vortexstrength_\epsilon
		- \vortexstrength_\epsilon\log\frac{1}{\epsilon}\bigg(
			(1-\delta)\tau + (1-\eta)\frac{(1-\tau)^2}{\tau} + \kappa
		\bigg)
		\\- \vortexstrength_\epsilon\log\frac{1}{\epsilon}
			\sup_\domain\bigg\{\frac{8\pi\externalflowcoefficient_{\epsilon}}{(\sup_\domain\depth)\vortexstrength_\epsilon\log\frac{1}{\epsilon}}|\flow| \bigg\}
	\end{multline*}
Using the fact that
	\[ (1-\delta)\tau + (1-\eta)\frac{(1-\tau)^2}{\tau} + \kappa
		+ \sup_\domain\bigg\{\frac{8\pi\externalflowcoefficient_{\epsilon}}{(\sup_\domain\depth)\vortexstrength_\epsilon\log\frac{1}{\epsilon}}|\flow|\bigg\}
	< \frac{s\tau}{2} ,\]
we obtain the strict lower bound
	\[ \tau\vortexstrength_\epsilon\int_\domain\log\frac{\epsilon}{|x-y|}\ \dif\positivepart{\vortex_\epsilon}(y)
		> - \ncste\vortexstrength_\epsilon - \frac{s\tau}{2}\vortexstrength_\epsilon\log\frac{1}{\epsilon} .\]
Fix $R\geq 1$. Then using the same computation than the one in the proof of \cref{firstorderflowupperbound}, we infer the existence of
some constant $\ncste>0$ such that, for all $x\in\domain$:
	\[ \int_{B(x,R\epsilon)}\log\frac{\epsilon}{|x-y|}\ \dif\positivepart{\vortex_\epsilon}(y)
		\leq \cste .\]
Injecting this upper bound in the above computation, we obtain for all $x\in D^\kappa_\epsilon$
and for all $R\geq 1$:
	\[ \tau\vortexstrength_\epsilon\int_{\domain\setminus B(x,R\epsilon)}\log\frac{\epsilon}{|x-y|}\ \dif\positivepart{\vortex_\epsilon}(y)
		> - \ncste\vortexstrength_\epsilon - \frac{s\tau}{2}\vortexstrength_\epsilon\log\frac{1}{\epsilon} .\]
Thus for all $R> 1$ and for all $x\in D^\kappa_\epsilon$, we have
	\[ \tau\vortexstrength_\epsilon\int_{\domain\setminus B(x,R\epsilon)}\dif\positivepart{\vortex_\epsilon}(y)
		< \frac{\ncste\vortexstrength_\epsilon + \frac{s\tau}{2}\vortexstrength_\epsilon\log\frac{1}{\epsilon}}{\log(R)} .\]
Choose $R=\epsilon^{-s}$. Then we have for all $x\in D^\kappa_\epsilon$, and for all sufficiently small $\epsilon>0$:
	\[ \tau\vortexstrength_\epsilon\int_{\domain\setminus B(x,R\epsilon)}\dif\positivepart{\vortex_\epsilon}(y)
		< \frac{1}{2}\ \tau\vortexstrength_\epsilon .\]
Choosing $x_1,x_2\in D^\kappa_\epsilon$, we now observe that if $\distance(x_1,x_2)\geq R\epsilon=\epsilon^{1-s}$, then we would have
	\[ \tau\vortexstrength_\epsilon
		\leq \tau\vortexstrength_\epsilon\int_{\domain\setminus B(x_1,R\epsilon)}\dif\positivepart{\vortex_\epsilon}(y)
			+ \tau\vortexstrength_\epsilon\int_{\domain\setminus B(x_2,R\epsilon)}\dif\positivepart{\vortex_\epsilon}(y)
	< \tau\vortexstrength_\epsilon ,\]
which would be a contradiction. Thus any two point $x_1,x_2$ in $D^\kappa_\epsilon$ satisfy $\distance(x_1,x_2)\leq \epsilon^{1-s}$.
The claim now follows from \cref{vortextruncation}.
\end{proof}

\subsection{Proof of the main result}

From the following result, one may conclude \cref{theoremSteadyIntroduction},~page~\pageref{theoremSteadyIntroduction}:
\begin{theorem}\label{mainsteadystheorem}
If $\tau>0$, then every accumulation points as $\epsilon\to0$ of the family of probability measures $\big\{\dif\positivepart{\vortex_\epsilon}:\epsilon>0\big\}$,
is the sense of vague convergence of measures,
is a Dirac mass $\delta_{x^\star}$ with
	\[ \frac{\tau\depth(x^\star)}{4\pi} + \flow(x^\star)
		= \sup_{\domain}\bigg\{ \frac{\tau\depth}{4\pi} + \flow \bigg\} . \]
Similarly, if $\tau<1$, then every accumulation points as $\epsilon\to0$ of the family of probability measures $\big\{\dif\negativepart{\vortex_\epsilon}:\epsilon>0\big\}$,
in the sense of vague convergence of measures,
is a Dirac mass $\delta_{x_\star}$ with
	\[ \frac{(1-\tau)\depth(x_\star)}{4\pi} - \flow(x_\star)
		= \sup_{\domain}\bigg\{ \frac{(1-\tau)\depth}{4\pi} - \flow \bigg\} . \]
\end{theorem}
\begin{proof}
It is sufficient to prove the claim for the positive parts $\positivepart{\vortex_\epsilon}$ with $\tau>0$.
According to \cref{errortermdefinition}, there exists $\{\error_\epsilon:\epsilon>0\}$ such that, for sufficiently small $\epsilon>0$:
	\begin{multline*}
		\vortexstrength_\epsilon^2\log\frac{1}{\epsilon}\Bigg( \tau\sup_\domain\bigg\{
			\frac{\tau\depth}{4\pi} + \frac{\externalflowcoefficient_{\epsilon}}{\vortexstrength_\epsilon\log\frac{1}{\epsilon}}\flow
		\bigg\}
		+ (1-\tau)\sup_\domain\bigg\{
			\frac{(1-\tau)\depth}{4\pi}-\frac{\externalflowcoefficient_{\epsilon}}{\vortexstrength_\epsilon\log\frac{1}{\epsilon}}\flow
		\bigg\} \Bigg)
		\\ \leq \energy_\epsilon\big( \vortex_\epsilon \big)
			- \error_\epsilon\vortexstrength_\epsilon^2\log\frac{1}{\epsilon}.
	\end{multline*}
According to \cref{concentration}, there exists $\sigma,\kappa>0$ such that for all sufficiently small $\epsilon>0$, there exists a ball $B^+_\epsilon=B(x_\epsilon,\epsilon^\sigma)$
with $x_\epsilon\in\domain$ and such that
	\[ \int_{B^+_\epsilon}\dif\positivepart{\vortex_\epsilon} \leq 
		\frac{1}{\kappa}\frac{4\pi}{\sup_\domain\depth}\vortexstrength_\epsilon\Bigg( \error_\epsilon + \frac{C}{\log\frac{1}{\epsilon}} \Bigg) .\]
Let $x^\star\in\closure{\domain}$ be an accumulation point for $\{x_\epsilon:\epsilon>0\}$ as $\epsilon\to 0$,
and $(\epsilon_n)_{n\in\integers}$ the corresponding sequence, decreasing to $0$.
For all $\varphi\in C(\closure{\domain})$, we have from the construction of $x_\epsilon$:
	\[ \lim_{n\to +\infty}\Bigg| \varphi(x^\star) - \int_\domain\varphi\dif\positivepart{\vortex_\epsilon} \Bigg| = 0 .\]
Therefore the sequence of probability measures $\big( \positivepart{\vortex_{\epsilon_n}} \big)_{n\in\integers}$ converges in the sense of
vague convergence to $\delta_{x^\star}$. Now for the last part of the claim, we apply \cref{firstorderflowupperbound}
and the Green's function expansion to obtain
	\begin{multline*}
		\energy_\epsilon(\vortex_\epsilon)
		\leq \tau\vortexstrength_\epsilon\int_\domain\bigg(\frac{\tau\depth}{4\pi}\log\frac{1}{\epsilon}\vortexstrength_\epsilon + \externalflowcoefficient_{\epsilon}\flow\bigg)\dif\positivepart{\vortex_\epsilon}
		 \\ + (1-\tau)\vortexstrength_\epsilon^2\log\frac{1}{\epsilon}\sup_\domain\bigg\{\frac{(1-\tau)\depth}{4\pi} - \frac{\externalflowcoefficient_{\epsilon}}{\log\frac{1}{\epsilon}\vortexstrength_\epsilon}\flow\bigg\}
		+ \ncste\vortexstrength_\epsilon^2 .
	\end{multline*}
This gives the inequality
	\[ \sup_\domain\bigg\{
			\frac{\tau\depth}{4\pi} + \frac{\externalflowcoefficient_{\epsilon}}{\vortexstrength_\epsilon\log\frac{1}{\epsilon}}\flow
		\bigg\}
		- \int_\domain\bigg(\frac{\tau\depth}{4\pi} + \frac{\externalflowcoefficient_{\epsilon}}{\vortexstrength_\epsilon\log\frac{1}{\epsilon}}\flow\bigg)\dif\positivepart{\vortex_\epsilon}
		\leq  \frac{1}{\tau}\Big(|\error_\epsilon| + \frac{\ncste}{\log\frac{1}{\epsilon}}\Big) .\]
Letting $\epsilon\to 0$ and relying on compactness of $\closure{\domain}$ and continuity of $\depth$ and $\flow$,
we obtain
	\[ \sup_\domain\bigg\{ \frac{\tau\depth}{4\pi} + \flow	\bigg\}
		= \lim_{\epsilon\to 0}\int_\domain\bigg(\frac{\tau\depth}{4\pi} + \flow\bigg)\dif\positivepart{\vortex_\epsilon} . \qedhere \]
\end{proof}

\section{Construction of a rotating singular pair}

In this section we apply \cref{mainsteadystheorem} to construct a rotating singular pair in a lake $(B(0,1),\depth)$
invariant under rotations. 
On $\domain=B(0,1)$ we will consider a continuous radial depth function $\depth:\domain\to\reals$ with the property that $(\domain,\depth)$ is continuous.
Our computations may also be done in the case where $\domain$ is an annulus, the only required change to bring would be
the introduction of prescribed circulations on the inner boundary.

Let us fix a distribution function
	\[ D : \reals^+ \to [0,\mumeasure(\domain)] \]
normalized as
	\[ \int_{\reals^+}D(t)\dif t = 1 ,\]
and such that there exists $p>1$ with
	\[ \int_{\reals^+}t^pD(t)\dif t < +\infty . \]
We also fix $\tau\in[0,1]$. We say that
a family of measurable functions $\big\{\vortex_\epsilon:\epsilon>0\big\}$ satisfies the constraints~\eqref{DinText}
	\begin{equation}\label{DinText}\tag{D}
		\mumeasure\big( \{\positivepart{\vortex_\epsilon}\geq\lambda \} \big) = \frac{\epsilon^2}{\delta}D\bigg( \frac{\epsilon^2\lambda}{\delta\tau}\,\log\frac{1}{\epsilon} \bigg),
		\qquad
		\mumeasure\big( \{\negativepart{\vortex_\epsilon}\geq\lambda \} \big) = \frac{\epsilon^2}{\delta}D\bigg( \frac{\epsilon^2\lambda}{\delta(1-\tau)}\,\log\frac{1}{\epsilon} \bigg),
	\end{equation}
where $\delta = \sup\limits_{\lambda>0}D(\lambda)$. An adaptation of the proof of \cref{symmetrizearoundpoint} shows that
such family always exists. A straightforward computation also shows that such a family satisfies
	\[ \mu\big(\{\positivepart{\vortex_\epsilon}>0\}\big) \leq \epsilon^2,
		\qquad \mu\big(\{\negativepart{\vortex_\epsilon}>0\}\big) \leq \epsilon^2 ,\]
	\[ \int_\domain\positivepart{\vortex_\epsilon}\mumeasure = \frac{\tau}{\log\frac{1}{\epsilon}} ,
		\qquad \int_\domain\negativepart{\vortex_\epsilon}\mumeasure = \frac{1-\tau}{\log\frac{1}{\epsilon}} ,\]
and finally
	\[ \sup_{\epsilon>0}\bigg\{ \log\frac{1}{\epsilon}\epsilon^{2(1-\frac{1}{p})}
		\Big( \norm{\positivepart{\vortex_\epsilon}}{L^p_{\mumeasure}} + \norm{\negativepart{\vortex_\epsilon}}{L^p_{\mumeasure}} \Big) \bigg\} < +\infty . \]
We are thus in position to apply \cref{mainsteadystheorem}. Exploiting the symmetry of the domain, we construct a rotating vortex pair around the origin $(0,0)$:
\begin{theorem}[\Cref{theoremRotatingIntroduction},~p.~\pageref{theoremRotatingIntroduction}]\label{rotatingpair}
Assume that $\domain=B(0,1)$ and $\depth\in C(\closure{\domain})\cap W^{1,\infty}_{\text{loc}}(\domain)$ is a radial function: $b(x)=b(|x|^2)$,
such that $(\domain,\depth)$ is a continuous lake.

There exists $\nu_0>0$ depending only on $\tau\in[0,1]$ such that, for all $\nu\in[0,\nu_0)$,
there exists a family $\big\{ (\vortex_\epsilon,u_\epsilon)\in C(\reals,L^p(\domain,\mumeasure))\times C(\reals,L^2(\domain,\mumeasure)):\epsilon>0 \big\}$
that solves the weak form of
	\[ \begin{cases}
		\divergence( \depth u_\epsilon ) = 0 \\
		\curl( u_\epsilon ) = \depth\vortex_\epsilon \\
		\partial_t\vortex_\epsilon + \scalarproduct{u_\epsilon}{\gradient\vortex_\epsilon}{\plane} = 0 ;
	\end{cases} \]
the family $\big\{\vortex_\epsilon:\epsilon>0\big\}$ satisfies constraint~\eqref{distributionconstraint},
and for each $t\in\reals$, the potential vortex $\vortex_\epsilon(t)$ at time $t$ is obtained from $\vortex_\epsilon(0)$ by a rotation of clockwise
angle $\nu t$.

Furthermore, if $\tau>0$, the only possible accumulation points of $\{\positivepart{\vortex_\epsilon(0)}:\epsilon>0\}$ as $\epsilon\to 0$, in the sense of
convergence of probability measures on $\domain$,  are Dirac masses $\delta_{x^\star}$ with
	\[ \frac{\tau\depth(x^\star)}{4\pi} - \frac{\nu}{2}\int^{|x^\star|^2}_0\depth(s)\dif s
		= \sup_{z\in\domain}\bigg\{ \frac{\tau\depth(z)}{4\pi} -\frac{\nu}{2}\int^{|z|^2}_0\depth(s)\dif s \bigg\} ,\]
and if $\tau<1$, the only possible accumulation points of $\{\negativepart{\vortex_\epsilon}:\epsilon>0\}$ as $\epsilon\to 0$, in the sense of
convergence of probability measures on $\domain$,  are Dirac masses $\delta_{x_\star}$ with
	\[ \frac{(1-\tau)\depth(x_\star)}{4\pi} + \frac{\nu}{2}\int^{|x^\star|^2}_0\depth(s)\dif s
		= \sup_{z\in\domain}\bigg\{ \frac{(1-\tau)\depth(z)}{4\pi} + \frac{\nu}{2}\int^{|z|^2}_0\depth(s)\dif s \bigg\} .\]
\end{theorem}
Since the boundary of $\domain=B(0,1)$ is connected, the operator $\gradient\rectifycirculation$
is null, so that a weak formulation of the time-dependent evolution equation
	\[ \partial\vortex_t + \scalarproduct{u_t}{\gradient\vortex_t}{\plane} = 0 \]
reads as
	\[ \int_\reals\int_\domain\partial_t(\varphi(t,\cdot))\vortex_\epsilon(t)\mumeasure\dif t
		+ \int_\reals\int_\domain\scalarproduct{\flipgradient\vortexoperator(\vortex_\epsilon(t))}{\gradient\varphi(t,\cdot)}{\plane}\ \vortex_\epsilon\lebesguemeasure\dif t
		= 0,\]
for all $\varphi\in\smoothfunctions(\reals\times\domain)$.
\begin{proof}[Proof of \cref{rotatingpair}]
Define for all $x\in \closure{B(0,1)}$:
	\[ \flow(x) = \frac{\nu}{2}\int^{|x|^2}_{0}\depth(s)\dif s ,\]
with $\nu\in\reals$ such that
	\[ |\nu| < \frac{\min\{\tau,(1-\tau)\}}{4\pi} .\]
In particular, this implies that
	\[ \sup_\domain\bigg\{ \frac{4\pi|\flow|}{(\sup_\domain\depth)} \bigg\} < \frac{\min\{\tau,1-\tau\}}{2} .\]
For all $\epsilon>0$, consider the energy
	\[ \energy_\epsilon : L^p(\domain,\mumeasure)\to\reals:
		\energy_\epsilon(\vortex) = \frac{1}{2}\int_\domain\vortex\vortexoperator\vortex\mumeasure
			+ \vortexstrength_\epsilon\log\frac{1}{\epsilon}\int_\domain\flow\vortex\mumeasure .\]
According to \cref{mainsteadystheorem}, there exists a family $\{\vortex_\epsilon:\epsilon>0\}$ of solutions of the steady lake equation
with Coriolis external flow $\flow$, that maximizes $\energy_\epsilon$ over their set of own $\mumeasure$-rearrangements.
Thus for all $\phi\in\smoothfunctions(\domain)$, we have
	\[ \int_\domain\vortex_\epsilon\scalarproduct{\flipgradient(\vortexoperator\vortex_\epsilon+\flow)}{\gradient\phi}{\plane}\ \lebesguemeasure = 0 ,\]
or also, using the definition of $\flow$:
	\[ \int_\domain\vortex_\epsilon\scalarproduct{\flipgradient\vortexoperator\vortex_\epsilon - \nu x^\perp\depth(x)}{\gradient\phi}{\plane}\ \lebesguemeasure = 0 .\]
For all $\epsilon>0$ and for all $t\in\reals$, let us define
$\tilde{\vortex}_\epsilon(t) = \vortex_\epsilon\circ R_{t\nu}$,
where $R_{\alpha}$ is the notation of angle $\alpha$ in the plane $\plane$.
We have $\tilde{\vortex}_\epsilon\in C\big( \reals , \rearrangement(\vortex_\epsilon) \big)$,
for all $\epsilon>0$. Furthermore, we have
$\vortexoperator\tilde{\vortex}_\epsilon(t) = \vortexoperator\vortex_\epsilon\circ R_{t\nu}$.
A change of variables in space shows that
	\begin{multline*}
		\int_{\reals}\int_\domain\scalarproduct{\flipgradient\vortexoperator\tilde{\vortex}_\epsilon(t)(x)}{\gradient\varphi(t,x)}{\plane}\ \tilde{\vortex}_\epsilon(t,x)\lebesguemeasure(x)\dif t
			\\= \int_{\reals}\int_\domain\scalarproduct{\flipgradient\vortexoperator\vortex_\epsilon(x)}{\gradient_x(\varphi(t,\cdot)\circ R^{-1}_{t\nu})(x)}{\plane}\ \vortex_\epsilon\lebesguemeasure(x)\dif t 
			\\= -\nu\int_{\reals}\int_\domain\scalarproduct{x^\perp}{\gradient(\varphi(t,\cdot)\circ R^{-1}_{t\nu})(x)}{\plane}\ \vortex_\epsilon(x)\mumeasure(x)\dif t .
	\end{multline*}
Another change of variable in space and an integration by parts in time also give
	\begin{multline*}
		 \int_{\reals}\int_\domain \partial_t\varphi(t,x)\tilde{\vortex}_\epsilon(t,x)\mumeasure(x)\dif t
		 	= \int_\domain \int_{\reals}\vortex_\epsilon(x)(\partial_t\varphi)\big(t,R_{-t\nu}(x)\big)\dif t\mumeasure
		 	\\ = \nu\int_\domain \int_{\reals}\vortex_\epsilon(x)\scalarproduct{x^\perp}{\gradient_x(\varphi\circ R^{-1}_{t\nu})(x)}{\plane}\dif t\mumeasure(x) .
	\end{multline*}
Hence we have
	\[ \int_{\reals}\int_\domain \partial_t\varphi(t,x)\tilde{\vortex}_\epsilon(t,x)\mumeasure(x)\dif t
		+ \int_{\reals}\int_\domain\scalarproduct{\flipgradient\vortexoperator\tilde{\vortex}_\epsilon(t)(x)}{\gradient\varphi(t,x)}{\plane}
		\ \tilde{\vortex}_\epsilon(t,x)\lebesguemeasure(x)\dif t
	= 0 ,\]
for all test function $\varphi\in\smoothfunctions(\reals\times\domain)$ and for all $\epsilon>0$.
\end{proof}

\begin{example}
As an illustration of \cref{rotatingpair}, we compute the effect of
a angular rotation speed when $\tau=0.7$, and in the special case of $\domain=B(0,1)$
and $\depth(x)=P(|x|^2)$ with
	\[ P:[0,1]\to\reals^+:
		P(t) = 2-4\Big(t-\frac{1}{2}\Big)^2 . \]
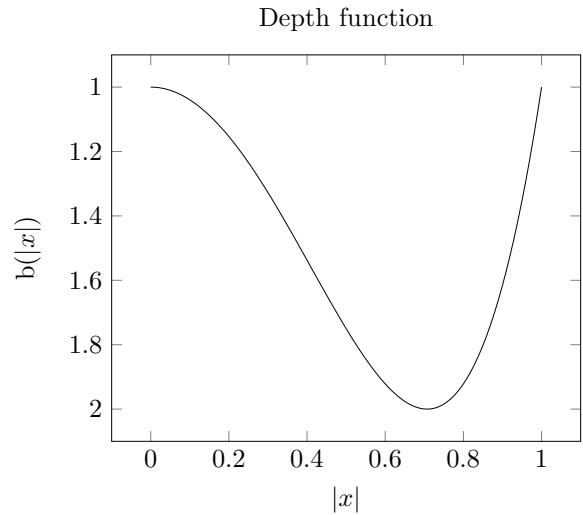
\begin{center}\small
\begin{tikzpicture}[scale=0.9]
\begin{axis}[title={Depth function},xlabel={$|x|$},ylabel={$\depth(|x|)$},y dir=reverse]
\addplot[color=black] table {./depth.txt};
\end{axis}
\end{tikzpicture}
\end{center}
We observe (figure~\ref{myFigure}) that when $\nu=0$, then
both parts are located near the point of maximum of $\depth$.
As $\nu$ increases, the rotation of the pair becomes more fast in the clockwise direction:
the positive part deviates in direction of the exterior, while the negative part comes closer to the disk center. This effect is reversed if $\nu$ as decreases.

\begin{figure}[!h]
\begin{center}\small
\begin{tikzpicture}[scale=0.9]
\begin{axis}[title={Depth against Angular Rotation Speed},xlabel={$\nu$},ylabel={$\depth(r)$}]
\addplot[color=red] table {./depth_mtP.txt};
\addplot[color=blue] table {./depth_mtN.txt};
\end{axis}
\end{tikzpicture}
\begin{tikzpicture}[scale=0.9]
\begin{axis}[title={Radius of location against Angular Rotation Speed},xlabel={$\nu$},ylabel={$r$},ylabel near ticks, yticklabel pos=right]
\addplot[color=red] table {./mtP.txt};
\addplot[color=blue] table {./mtN.txt};
\end{axis}
\end{tikzpicture}
\end{center}
\begin{center}
\caption{The effect of angular rotation speed $\nu$ when $\tau=0.7$
on the positive part (in red) and the negative part (in blue).}\label{myFigure}
\end{center}
\end{figure}
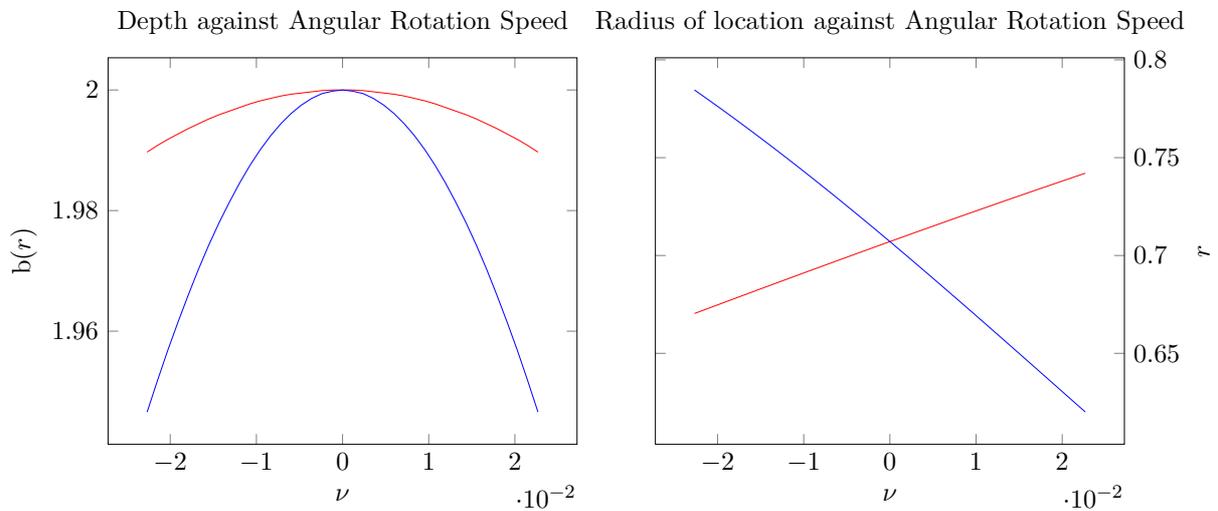
In a forthcoming paper, we investigate in more details the
expected distance between the two parts of the pair as $\epsilon\to 0$,
and the expected distance between the pair and the boundary of $\domain$.
\end{example}

\appendix

\section{Green's function expansion}\label{appendix}

In this section we prove our Green's function expansion, \cref{reduciblelake}; that is: we would like to prove that
the operator $\vortexoperator+\rectifycirculation$ admits the integral representation
	\[ \vortexoperator\vortex(x) + \rectifycirculation\vortex(x)
		= \depth(x)\int_\domain\greenlaplace(x,y)\vortex(y)\mumeasure(y) + \int_\domain\rectifykernel(x,y)\vortex(y)\mumeasure(y) \]
for some $\rectifykernel:\domain\times\domain\to\reals$ bounded and measurable, and $\greenlaplace$ the Green's function for the Laplace's operator $\laplacian$
on $\domain$ with Dirichlet boundary conditions.
Recall also that for a fixed $\vortex\in L^p(\domain,\mumeasure)$, the function $\vortexoperator(\vortex)\in\functionspace$ is the only function
in $\functionspace$ that solves the weak formulation of the elliptic equation
	\[ -\divergence\Big( \depth^{-1}\gradient\psi \Big) = \depth\vortex .\]
Before doing the proof, let us sketch the general idea.
For fixed $y\in\domain$, consider the ansatz
	\[ g^\sharp :\domain\to\reals : g^\sharp(x) = b(x)\greenlaplace(x,y) .\]
Then we have at a formal level
	\[ \gradient g^\sharp = b\gradient\greenlaplace(\cdot,y) + \greenlaplace(\cdot,y)\gradient\depth ,\]
and therefore
	\[ -\divergence\Big( \depth^{-1}\gradient g^\sharp \Big) = \laplacian\greenlaplace(\cdot,y) - \divergence\Big(\greenlaplace(\cdot,y)\depth^{-1}\gradient\depth \Big)
		= \delta_y - \divergence\Big(\greenlaplace(\cdot,y)\depth^{-1}\gradient\depth \Big) .\]
It is then natural to look for $\rectifykernel(\cdot,y)\in\functionspace$ such that
	\[ -\divergence\Big( \depth^{-1}\gradient\rectifykernel(\cdot,y) \Big) = \divergence\Big(\greenlaplace(\cdot,y)\depth^{-1}\gradient\depth \Big) ,\]
whose weak form is
	\[ \int_\domain\scalarproduct{\gradient\rectifykernel(\cdot,y)}{\gradient\varphi}{\plane}\ \frac{\lebesguemeasure}{\depth}
		= -\int_\domain\scalarproduct{\greenlaplace(\cdot,y)\gradient\depth}{\gradient\varphi}{\plane}\ \frac{\lebesguemeasure}{\depth} ,\qquad\forall\varphi\in\functionspace .\]
\begin{proposition}\label{boundednessregularity}
Let $(\domain,\depth)$ be a continuous lake.
For all $y\in\domain$, there exists a unique function $\rectifykernel(\cdot,y)\in\functionspace$ such that
	\[ \int_\domain\scalarproduct{\gradient\rectifykernel(\cdot,y)}{\gradient\varphi}{\plane}\ \frac{\lebesguemeasure}{\depth}
		= -\int_\domain\scalarproduct{\greenlaplace(\cdot,y)\gradient\depth}{\gradient\varphi}{\plane}\ \frac{\lebesguemeasure}{\depth},\qquad\forall\varphi\in\functionspace .\]
Moreover, the function $\rectifykernel:\domain\times\domain\to\reals:(x,y)\mapsto\rectifykernel(x,y)$
is measurable and bounded on $\domain\times\domain$.
\end{proposition}
We mimic the proof of \cite{GilbardTrudinger}*{Theorem~8.8}.
\begin{proof}
We recall that since $(\domain,\depth)$ is continuous, by \cref{continuouslakes}, there exists $\ell>2$ such that function
	\[ y\in\domain \mapsto \int_\domain g(\cdot,y)^\ell \ \bigg(\frac{|\gradient b|^2}{b}\bigg)^{\frac{\ell}{2}}\ \lebesguemeasure \]
is bounded on $\domain$, and for all $y_\star\in\domain$, we have
	\[ \lim_{y\to y_\star}\int_\domain |g(\cdot,y)-g(\cdot,y_\star)|^\ell \ \bigg(\frac{|\gradient b|^2}{b}\bigg)^{\frac{\ell}{2}}\ \lebesguemeasure = 0 . \]
First define
	\[ \kappa = 2\sup_{y\in\domain}\bigg\{
		\int_\domain \greenlaplace(\cdot,y)^\ell\bigg(\frac{|\gradient\depth|^2}{\depth}\bigg)^{\frac{\ell}{2}}\ \lebesguemeasure
	\bigg\}^{\frac{1}{\ell}} .\]
From the Hölder's inequality, since $\domain$ is bounded by assumption that $(\domain,\depth)$ is a lake, we have
	\[ \Bigg| \int_\domain\greenlaplace(\cdot,y)^2\ \frac{|\gradient\depth|^2}{\depth}\ \lebesguemeasure \Bigg|^{\frac{1}{2}} \leq \ncste\kappa ,\]
for some $\cste>0$, and therefore the linear functional
	\[ \varphi\in\functionspace \mapsto -\int_\domain\scalarproduct{\greenlaplace(\cdot,y)\gradient\depth}{\gradient\varphi}{\plane}\ \frac{\lebesguemeasure}{\depth} \]
is continuous on $\functionspace$. According to Riesz's representation theorem, there exists a unique function $\rectifykernel(\cdot,y)\in\functionspace$
in $\functionspace$ such that
	\[ \int_\domain\scalarproduct{\gradient\rectifykernel(\cdot,y)}{\gradient\varphi}{\plane}\ \frac{\lebesguemeasure}{\depth}
		= -\int_\domain\scalarproduct{\greenlaplace(\cdot,y)\gradient\depth}{\gradient\varphi}{\plane}\ \frac{\lebesguemeasure}{\depth} ,\qquad\forall\varphi\in\functionspace ,\]
and $\norm{\rectifykernel(\cdot,y)}{\functionspace}\leq \cste\kappa$.
Let us now prove that $\rectifykernel(\cdot,y)$ is bounded independently of $y\in\domain$.
For all $\beta\geq 1$ and for all $N>\kappa$,
define the auxiliary functions
	\[ F : \reals^+\to\reals : F(t) = \max\big\{ N , t^\beta \big\} ,\]
and
	\[ G: \reals\to\reals : G(t) = \int_{\kappa}^{\max\{t,\kappa\}}|F'(s)|^2\dif s ,\]
so that for all $t\in\reals$, we have
	\[ G'(t) = \chi_{t\geq \kappa}|F'(t)|^2 = \chi_{[\kappa,N]}(t)\ \beta^2t^{2(\beta-1)} .\]
Thus $G$ is derivable with bounded derivative, and $G(t)=0$ for all $t\leq\kappa$.
Define the positive function $u=\max\big\{\rectifykernel(\cdot,y),\kappa\big\}$. Then the function $G\circ u$ belongs to $\functionspace$ with
	\[ \gradient (G\circ u) = \gradient u\ (G'\circ u) = \chi_{\{\rectifykernel(\cdot,y)\geq\kappa\}}\gradient \rectifykernel(\cdot,y)\ (G'\circ u) .\]
From this it follows that
	\[ \int_\domain\scalarproduct{\gradient u}{\gradient (G\circ u)}{\plane}\ \frac{\lebesguemeasure}{\depth}
		= -\int_\domain \scalarproduct{\greenlaplace(\cdot,y)\gradient\depth}{\gradient u}{\plane}\ (G'\circ u)\frac{\lebesguemeasure}{\depth} .\]
Since $\scalarproduct{\gradient u}{\gradient (G\circ u)}{\plane}=|\gradient(F\circ u)|^2$ is positive, one may apply
Schwartz' and Young's inequalities to get
	\[ \int_\domain\scalarproduct{\gradient u}{\gradient (G\circ u)}{\plane}\ \frac{\lebesguemeasure}{\depth}
		\leq 4\int_\domain \greenlaplace(\cdot,y)^2\frac{|\gradient\depth|^2}{\depth}\ (G'\circ u)\lebesguemeasure .\]
Recalling that $u\geq\kappa\geq 0$ by construction, we have
	\[ \int_\domain\scalarproduct{\gradient u}{\gradient (G\circ u)}{\plane}\ \frac{\lebesguemeasure}{\depth}
		\leq \frac{4}{\kappa^2}\int_\domain \greenlaplace(\cdot,y)^2\frac{|\gradient\depth|^2}{\depth}\ |u|^2(G'\circ u)\lebesguemeasure ,\]
or also, in terms of the auxiliary function $F$:
	\[ \int_\domain|\gradient (F\circ u)|^2\ \frac{\lebesguemeasure}{\depth}
		\leq \frac{4}{\kappa^2}\int_\domain \greenlaplace(\cdot,y)^2\frac{|\gradient\depth|^2}{\depth}\ |u(F'\circ u)|^2\lebesguemeasure .\]
By Hölder's inequality
	\[ \int_\domain|\gradient (F\circ u)|^2\ \frac{\lebesguemeasure}{\depth}
		\leq \frac{4}{\kappa^2}\bigg(\int_\domain \greenlaplace(\cdot,y)^\ell\bigg(\frac{|\gradient\depth|^2}{\depth}\bigg)^{\frac{\ell}{2}}\ \lebesguemeasure\bigg)^{\frac{2}{\ell}}
			\norm{u(F'\circ u)}{L^{\frac{2\ell}{\ell-2}}_{\lebesguemeasure}}^2 .\]
By construction of $\kappa$, we have
	\[ \int_\domain|\gradient (F\circ u)|^2\ \frac{\lebesguemeasure}{\depth}
		\leq \norm{u(F'\circ u)}{L^{\frac{2\ell}{\ell-2}}_{\lebesguemeasure}}^2 .\]
Since $b$ is bounded and $\domain$ is bounded, one may use Hölder's inequality and Sobolev's inequality (because $F\circ u\in W^{1,2}_0(\domain)$)
to show that there exists some constant $\ncste>0$ such that
	\[ \norm{F\circ u}{L_{\lebesguemeasure}^{2\,\frac{2\ell}{\ell-2}}}
		\leq \cste\norm{u(F'\circ u)}{L^{\frac{2\ell}{\ell-2}}_{\lebesguemeasure}} .\]
Using the definition of $F$ and letting $N\to +\infty$, we obtain for all $\beta\geq 1$:
	\[ \norm{u^\beta}{L_{\lebesguemeasure}^{2\,\frac{2\ell}{\ell-2}}}
		\leq \beta\cste\norm{u^\beta}{L^{\frac{2\ell}{\ell-2}}_{\lebesguemeasure}} .\]
One may now use the technique of iteration of norms to prove that there exists some constant $\ncste>0$ such that
	\[ \norm{u}{L^\infty} \leq \cste\norm{u}{L^{\frac{2\ell}{\ell-2}}_{\lebesguemeasure}} \leq \ncste\big( \kappa + \norm{\rectifykernel(\cdot,y)}{\functionspace} \big)
		\leq \ncste\kappa . \]
This implies, by construction of $u$, that $|\rectifykernel(\cdot,y)|\leq \max\{1,\cste\}\kappa$.
Observe that by symmetry of the absolute value, this estimate also hold for $-\rectifykernel(\cdot,y)$, so that
	\[ \norm{\rectifykernel(\cdot,y)}{L^\infty} \leq \ncste\bigg(\int_\domain \greenlaplace(\cdot,y)^\ell\bigg(\frac{|\gradient\depth|^2}{\depth}\bigg)^{\frac{\ell}{2}}\ \lebesguemeasure\bigg)^{\frac{1}{\ell}} ,\]
for all $y\in\domain$. Because the lake $(\domain,\depth)$ is continuous, the function
	\[ \rectifykernel : \domain\times\domain\to\reals : (x,y)\mapsto \rectifykernel(x,y) \]
is uniformly bounded. Furthermore, for fixed $y\in\domain$, the function
$\big[x\in\domain\mapsto \rectifykernel(x,y)\big]$
is measurable (by construction) and for fixed $x\in\domain$, the function
$\big[y\in\domain\mapsto \rectifykernel(x,y)\big]$
is continuous. Indeed, by linearity of the elliptic equation and according to the above estimate, we have
for all points $y,y^\star\in\domain$:
	\[ \big| \rectifykernel(x,y) - \rectifykernel(x,y^\star) \big|
		\leq \sup_{z\in\domain}\big| \rectifykernel(z,y) - \rectifykernel(z,y^\star) \big|
			\leq \ncste \bigg(\int_\domain |\greenlaplace(\cdot,y)-\greenlaplace(\cdot,y^\star)|^\ell\bigg(\frac{|\gradient\depth|^2}{\depth}\bigg)^{\frac{\ell}{2}} \lebesguemeasure\bigg)^{\frac{1}{\ell}} ,\]
which goes to $0$ as $y\to y^\star$. In particular, the function $\rectifykernel$ is measurable on $\domain\times\domain$.
\end{proof}
\begin{remark}
Let $\varphi\in C^1_c(\domain)$ and consider the solution of the problem
	\[ -\divergence\big( \depth^{-1}\gradient \psi \big) = -\divergence\big( \varphi\depth^{-1}\gradient\depth \big) .\]
Then using standard regularity theory~\cite{GilbardTrudinger}*{Theorem~8.22}, one can prove that the solution $\psi_\varphi$ of this problem
is continuous. In particular, if we let $\varphi$ goes to $\greenlaplace(\cdot,y)$ for some fixed $y\in\domain$,
in such a way that
	\[ \bigg(\int_\domain |\varphi-\greenlaplace(\cdot,y)|^\ell\bigg(\frac{|\gradient\depth|^2}{\depth}\bigg)^{\frac{\ell}{2}} \lebesguemeasure\bigg)^{\frac{1}{\ell}} \to 0,
		\quad \varphi\to\greenlaplace(\cdot,y) ,\]
then we obtain that $\psi_\varphi\to\rectifykernel(\cdot,y)$ uniformly, and the function $\rectifykernel(\cdot,y)$ is continuous on $\domain$.
Now using the uniform continuity of the function $\rectifykernel(x,\cdot)$, for fixed $x\in\domain$, we obtain the continuity of the function
$\rectifykernel$ on the product $\domain\times\domain$.
\end{remark}
We are now ready to prove \cref{reduciblelake}, which we recall now:
\begin{theorem*}[\Cref{reduciblelake},~page~\pageref{reduciblelake}]
Let $(\domain,\depth)$ be a continuous lake.
There exists a bounded measurable function $\rectifykernel:\domain\times\domain\to\reals$
such that, for all $\vortex\in L^p(\domain,\mumeasure)$, $p>1$, we have almost-everywhere
	\[ \vortexoperator\vortex(x) + \rectifycirculation\vortex(x)
		= \depth(x)\int_\domain\greenlaplace(x,y)\vortex(y)\mumeasure(y) + \int_\domain\rectifykernel(x,y)\vortex(y)\mumeasure(y) .\]
Furthermore, we have $\rectifykernel(\cdot,y)\in\functionspace$
for all $y\in\domain$, with for all $\varphi\in\functionspace$:
	\[ \int_\domain\scalarproduct{\gradient\rectifykernel(\cdot,y)}{\gradient\varphi}{\plane}\ \frac{\lebesguemeasure}{\depth}
	= -\int_\domain\scalarproduct{\greenlaplace(\cdot,y)\gradient\depth}{\gradient\varphi}{\plane}\ \frac{\lebesguemeasure}{\depth} .\]
\end{theorem*}
\begin{proof}
Let $\rectifykernel$ be as given by \cref{boundednessregularity}. Then $\rectifykernel$ is bounded and measurable.
Without loss of generality, one may assume that $\vortex\in L^\infty(\domain)$.
Define
	\[ \psi : \domain\to\reals : \psi(x) = \int_\domain\greenlaplace(x,y)\vortex(y)\mumeasure(y) .\]
The function $\psi$ belongs to $W^{1,2}_0(\domain)\cap L^\infty(\domain)$.
Let $\varphi\in\smoothfunctions(\domain)$, and let $(\psi_n)_{n\in\integers}$ be a sequence of $\smoothfunctions(\domain)$ functions
converging in $W^{1,2}_0(\domain)$ and uniformly to the function $\psi$. We have
	\[ \int_\domain\depth\psi_n\gradient\varphi\lebesguemeasure
		= -\int_\domain(\gradient\depth)\psi_n\varphi\lebesguemeasure - \int_\domain\depth\varphi\gradient\psi_n\lebesguemeasure ,\]
and letting $n\to +\infty$ yields the chain rule formula
	\[ \gradient(\depth\psi) = \depth\gradient\psi + \psi\gradient\depth .\]
On the other hand, for all $y\in\domain$, the function
	\[ x\in\domain\mapsto \gradient_x\rectifykernel(x,y) = \begin{pmatrix}
		\displaystyle\limsup_{t\to 0}\frac{\rectifykernel(x+(t,0),y) - \rectifykernel(x,y)}{t} \\
		\displaystyle\limsup_{t\to 0}\frac{\rectifykernel(x+(0,t),y) - \rectifykernel(x,y)}{t}
	\end{pmatrix} \]
almost-everywhere equals $\gradient\rectifykernel(\cdot,y)$~\cite{Evans}*{Theorem~2,~section~4.9.2}.
Hence the function $(x,y)\mapsto\gradient_x\rectifykernel(x,y)$
is Lebesgue-integrable on the product $\domain\times\domain$ (endowed with the Lebesgue's measure), by Tonelli's theorem and \cref{boundednessregularity}.
From this it follows that the ansatz function
	\[ \phi : \domain\to\reals: \phi(x) = \depth(x)\psi(x) + \int_\domain\rectifykernel(x,y)\vortex(y)\mumeasure(y) \]
is weakly differentiable with weak gradient given by
	\[ \depth\gradient\psi + \psi\gradient\depth + \int_\domain\gradient_x\rectifykernel(x,y)\vortex(y)\mumeasure(y) .\]
Now we compute from the definitions
	\[ \int_\domain\scalarproduct{\gradient\phi}{\gradient\varphi}{\plane}\ \frac{\lebesguemeasure}{\depth}
		= \int_\domain\varphi\vortex\mumeasure ,\]
so that $\phi\in W^{1,2}(\domain)$ is a weak solution for the same problem as $\vortexoperator(\vortex)$.
It remains to prove that $\phi$ belongs to $\functionspace$.
First, let us prove that it belongs to $\largefunctionspace$. Since it already belongs to $W^{1,2}(\domain)$,
it remains to observe that
	\[ \frac{|\gradient\phi|^2}{\depth} \leq \ncste\Bigg( \depth|\gradient\psi|^2 + |\psi|^2\frac{|\gradient\depth|^2}{\depth}
		+ \depth^{-1}\bigg|\int_\domain\gradient_x\rectifykernel(x,y)\vortex(y)\mumeasure(y)\bigg|^2 \Bigg) \]
in integrable on $\domain$.
The last term is integrable, according to \cref{boundednessregularity}.
The first term is integrable, by definition of $\psi$. The second term is integrable, according to
Hölder's inequality and because $(\domain,\depth)$ is a continuous lake. Therefore $\phi\in\largefunctionspace$.
Let us now prove that $\depth\psi\in\largefunctionspace$ belongs to $\functionspace$.
Taking back our sequence $(\varphi_n)_{n\in\integers}$, we use Hölder's inequality to observe that the sequence
$(\depth\varphi_n)_{n\in\integers}$ converges to $\depth\psi$ in $\largefunctionspace$. Since it belongs to $\functionspace$,
we have $\depth\psi\in\functionspace$. Now $\phi$ belongs to $\functionspace$. To see this, let us exploit the Hilbert structure of
$\largefunctionspace$ and consider a function $u\in(\functionspace)^\perp$ in the orthogonal space of $\functionspace\subseteq\largefunctionspace$.
Then we have, because $\rectifykernel(\cdot,y)\in\functionspace$ and by Fubini's theorem:
	\[ \scalarproduct{\phi}{u}{\largefunctionspace}
		= \int_\domain\scalarproduct{\rectifykernel(\cdot,y)}{u}{\largefunctionspace}\,\vortex(y)\mumeasure(y) = 0 ,\]
so that $\phi\in\big((\functionspace)^\perp\big)^\perp=\functionspace$. This yields $\phi=\vortexoperator(\vortex)$
whenever $\vortex\in L^\infty(\domain)$. Now both $\vortexoperator$ and the integral representation for $\phi$
are stable under limits in $L^p(\domain,\mumeasure)$, hence the conclusion.
\end{proof}

\section*{Acknowledgments}
The author warmly thanks his supervisor Jean~Van~Schaftingen for his kindness and his patience,
and for the many relevant discussions they had
during the preparation of this text.

\end{document}